\documentclass[12pt]{amsart}%
\usepackage{amsmath}
\usepackage{amsfonts}
\usepackage{amssymb}
\usepackage{amstext}
\usepackage{graphicx}%
\providecommand{\U}[1]{\protect\rule{.1in}{.1in}}
\providecommand{\U}[1]{\protect \rule{.1in}{.1in}}
\providecommand{\U}[1]{\protect \rule{.1in}{.1in}}
\newtheorem{theorem}{Theorem}[section]
\newtheorem{lemma}{Lemma}[section]
\newtheorem{proposition}{Proposition}[section]

\newtheorem{definition}{Definition}[section]

\numberwithin{equation}{section}
\newtheorem {conjecture}{Conjecture}

\theoremstyle{remark}
\newtheorem{remark}{Remark}[section]
\numberwithin{equation}{section}
\begin{document}
\title[Existence of nonconstant CR-holomorphic functions of polynomial growth]{Existence of nonconstant CR-holomorphic functions of polynomial growth in
Sasakian Manifolds}
\author{$^{\ast}$Shu-Cheng Chang}
\address{Department of Mathematics and Taida Institute for Mathematical Sciences
(TIMS), National Taiwan University, Taipei 10617, Taiwan}
\email{scchang@math.ntu.edu.tw }
\author{$^{\dag}$Yingbo Han}
\address{{School of Mathematics and Statistics, Xinyang Normal University}\\
Xinyang,464000, Henan, China}
\email{{yingbohan@163.com}}
\author{$^{\ddag}$Nan Li }
\address{Department of Mathematics, The City University of New York - NYC College of
Technology, Brooklyn, NY 11201, USA}
\email{NLi@citytech.cuny.edu}
\author{Chien Lin}
\address{Yau Mathematical Science Center, Tsinghua University, Haidian District,
Beijing 100084, China}
\email{clin@mail.tsinghua.edu.cn}
\thanks{$^{\ast}$Shu-Cheng Chang is partially supported in part by the MOST of Taiwan.}
\thanks{$^{\dag}$Yingbo Han is partially supported by an NSFC 11971415 and Nanhu
Scholars Program for Young Scholars of {Xinyang Normal University}.}
\thanks{$^{\ddag}$Nan Li is  partially supported by PSC-CUNY Grants 61533-0049.}

\begin{abstract}
In this paper, we show that there exists a nonconstant CR holomorphic function
of polynomial growth in a complete noncompact Sasakian manifold of nonnegative
pseudohermitian bisectional curvature with the CR maximal volume growth
property. This is the very first step toward the CR analogue of Yau
uniformization conjecture which states that any complete noncompact Sasakian
manifold of positive pseudohermitian bisectional curvature is CR biholomorphic
to the standard Heisenberg group.

\end{abstract}
\maketitle

\section{Introduction}

In K\"{a}hler geometry, Yau proposed a variety of uniformization-type problems
on complete noncompact K\"{a}hler manifolds with nonnegative holomorphic
bisectional curvature. The first Yau's uniformization conjecture is that

\begin{conjecture}
\label{conj1} If $M$ is a complete noncompact $m$-dimensional K\"{a}hler
manifold with nonnegative holomorphic bisectional curvature, then
\[
\dim_{%
%TCIMACRO{\U{2102} }%
%BeginExpansion
\mathbb{C}
%EndExpansion
}\left(  \mathcal{O}_{d}\left(  M^{m}\right)  \right)  \leq\dim_{%
%TCIMACRO{\U{2102} }%
%BeginExpansion
\mathbb{C}
%EndExpansion
}\left(  \mathcal{O}_{d}\left(
%TCIMACRO{\U{2102} }%
%BeginExpansion
\mathbb{C}
%EndExpansion
^{m}\right)  \right)  .
\]
The equality holds if and only if $M$ is isometrically biholomorphic to $%
%TCIMACRO{\U{2102} }%
%BeginExpansion
\mathbb{C}
%EndExpansion
^{m}$. Here $\mathcal{O}_{d}\left(  M^{m}\right)  $ denotes the family of all
holomorphic functions on a complete $m$-dimensional K\"{a}hler manifold $M$ of
polynomial growth of degree at most $d$.
\end{conjecture}

In \cite{n1}, Ni established the validity of this conjecture by deriving the
monotonicity formula for the heat equation under the assumption that $M$ has
maximal volume growth
\[
\lim_{r\rightarrow+\infty}\frac{Vol\left(  B\left(  p,r\right)  \right)
}{r^{2m}}\geq\alpha
\]
for a fixed point $p$ and a positive constant $\alpha$. Later, in \cite{cfyz},
the authors improved Ni's result without the assumption of maximal volume
growth. In recent years, G. Liu (\cite{liu1}) generalized the sharp dimension
estimate by only assuming that $M$ admits nonnegative holomorphic sectional curvature.

The second Yau's uniformization conjecture is that

\begin{conjecture}
\label{conj2} If $M$ is a complete noncompact $m$-dimensional K\"{a}hler
manifold with nonnegative holomorphic bisectional curvature, then the ring
$\mathcal{O}_{P}\left(  M\right)  $ of all holomorphic functions of polynomial
growth is finitely generated.
\end{conjecture}

This one was solved completely by G. Liu (\cite{liu2}) quite recently. He
mainly deployed techniques to attack this conjecture via Cheeger-Colding
(\cite{chco1}, \cite{chco2}), methods of heat flow developed by Ni and Tam
(\cite{n1}, \cite{nt1}, \cite{nt4}), H\"{o}rmander $L^{2}$-estimate of
$\overline{\partial}$ (\cite{de}) and three-circle theorem (\cite{liu1}) as well.

The third Yau's uniformization conjecture is that

\begin{conjecture}
\label{conj3} If $M$ is a complete noncompact $m$-dimensional K\"{a}hler
manifold with positive holomorphic bisectional curvature, then $M$ is
biholomorphic to the standard $m$-dimensional complex Euclidean space $%
%TCIMACRO{\U{2102} }%
%BeginExpansion
\mathbb{C}
%EndExpansion
^{m}$.
\end{conjecture}

The first giant progress relating to the third conjecture could be attributed
to Siu-Yau and Mok. In their papers (\cite{msy}, \cite{m1}, \cite{m2}), they
showed that, under the assumptions of the maximum volume growth condition and
the scalar curvature decays in certain rate, a complete noncompact
$m$-dimensional K\"{a}hler manifold $M$ with nonnegative holomorphic
bisectional curvature is isometrically biholomorphic to $%
%TCIMACRO{\U{2102} }%
%BeginExpansion
\mathbb{C}
%EndExpansion
^{m}$. A Riemannian version was solved in \cite{gw2} shortly afterwards. Since
then there are several further works aiming to prove the optimal result and
reader is referred to \cite{m2}, \cite{ctz}, \cite{cz}, \cite{n2}, \cite{nt1},
\cite{nt2} and \cite{nst}. For example, A. Chau and L. F. Tam (\cite{ct})
proved that a complete noncompact K\"{a}hler manifold with bounded nonnegative
holomorphic bisectional curvature and maximal volume growth is biholomorphic
to $%
%TCIMACRO{\U{2102} }%
%BeginExpansion
\mathbb{C}
%EndExpansion
^{m}$. Recently, G. Liu (\cite{liu3}) confirmed Yau's uniformization
conjecture when $M$ has maximal volume growth. Later, M.-C. Lee and L.-F. Tam
(\cite{lt}) also confirmed Yau's uniformization conjecture with the maximal
volume growth condition.

A Sasakian manifold is a strictly pseudoconvex CR $(2n+1)$-manifold of
vanishing pseudohermitian torsion which is an odd dimensional counterpart of
K\"{a}hler geometry. It is known (\cite{hs}) that a simply connected closed
Sasakian manifold with positive pseudohermitian bisectional curvature is CR
equivalent to the standard CR sphere $S^{2n+1}$. So it is very natural to
concerned with an CR analogue of Yau uniformization conjectures in a complete
noncompact Sasakian $(2n+1)$-manifold with positive pseudohermitian
bisectional curvature. More precisely, a smooth complex-valued function on a
strictly pseudoconvex CR $\left(  2n+1\right)  $-manifold $\left(
M,J,\theta\right)  $ is called CR-holomorphic if%
\[
\overline{\partial}_{b}f=0.
\]
For any fixed point $x\in M$, a CR-holomorphic function $f$ is called to be of
polynomial growth if there are a nonnegative number $d$ and a positive
constant $C=C\left(  x,d,f\right)  $, depending on $x$, $d$ and $f$, such that%
\[
\left\vert f\left(  y\right)  \right\vert \leq C\left(  1+d_{cc}\left(
x,y\right)  \right)  ^{d}%
\]
for all $y\in M$, where $d_{cc}\left(  x,y\right)  $ denotes the
Carnot-Caratheodory distance between $x$ and $y.$ In the following, we
sometimes would use the notation $r\left(  x,y\right)  $ for the
Carnot-Caratheodory distance. In fact, the definition above is independent of
the choice of the point $x\in M$. Finally we denote $\mathcal{O}_{d}^{CR}(M)$
the family of all CR-holomorphic functions $f$ of polynomial growth of degree
at most $d$ with $Tf(x)=f_{0}(x)=0:$
\[
\mathcal{O}_{d}^{CR}(M)=\{f(x)\ |\overline{\partial}_{b}f(x)=0,\ f_{0}%
(x)=0\text{ \textrm{and } }|f(x)|\leq C\left(  1+d_{cc}\left(  x,y\right)
\right)  ^{d}\text{ }\ \}.
\]

In the recent paper, we affirmed the first
%TCIMACRO{\QTR{frametitle}{CR Yau uniformization conjecture on Sasakian
%manifolds as following.}}%
%BeginExpansion
\title{CR Yau uniformization conjecture on Sasakian manifolds as
following.}%
%EndExpansion

\begin{proposition}
\label{P1} (\cite{chl1}, \cite{chl2}) If $(M,J,\theta)$ is a complete
noncompact Sasakian $(2n+1)$-manifold of nonnegative pseudohermitian
bisectional curvature, then%
\[
\dim_{%
%TCIMACRO{\U{2102} }%
%BeginExpansion
\mathbb{C}
%EndExpansion
}\left(  \mathcal{O}_{d}^{CR}\left(  M\right)  \right)  \leq\dim_{%
%TCIMACRO{\U{2102} }%
%BeginExpansion
\mathbb{C}
%EndExpansion
}\left(  \mathcal{O}_{d}^{CR}\left(  \mathbf{H}_{n}\right)  \right)
\]
with equality for some positive integer $d$ if and only if $M^{2n+1}$ is CR
equivalent to $\mathbf{H}_{n}.$ Here $\mathbf{H}_{n}$ = $%
%TCIMACRO{\U{2102} }%
%BeginExpansion
\mathbb{C}
%EndExpansion
^{n}\times%
%TCIMACRO{\U{211d} }%
%BeginExpansion
\mathbb{R}
%EndExpansion
$ is the $(2n+1)$-dimensional Heisenberg group.
\end{proposition}

Now we state the second and third CR Yau uniformization conjectures :

\begin{conjecture}
\label{conj4} If $M$ is a complete noncompact Sasakian $(2n+1)$-manifold of
nonnegative peudohermitian bisectional curvature, then the ring $\mathcal{O}%
_{p}^{CR}(M)$ of all CR holomorphic functions of polynomial growth is finitely generated.
\end{conjecture}

\begin{conjecture}
\label{conj5} If $M$ is a complete noncompact Sasakian $(2n+1)$-manifold of
positive pseudohermitian bisectional curvature. Then $M$ is CR biholomorphic
to the standard Heisenberg group%
\[
\mathbf{H}_{n}=%
%TCIMACRO{\U{2102} }%
%BeginExpansion
\mathbb{C}
%EndExpansion
^{n}\times%
%TCIMACRO{\U{211d} }%
%BeginExpansion
\mathbb{R}
%EndExpansion
.
\]

\end{conjecture}

In our previous papers, by applying the linear trace version of
Li-Yau-Hamilton inequality for positive solutions of the CR
Lichnerowicz-Laplacian heat equation and the CR moment type estimate of the CR
heat equation, we are able to obtain the following CR gap Theorem without
using maximum volume growth condition (\cite{n3}, \cite{cf}, \cite{ccf}, and
\cite{cchl}). In fact, let $M$ be a complete noncompact Sasakian
$(2n+1)$-manifold with nonnegative pseudohermitian bisectional curvature. Then
$M$ is CR flat if
\[
\frac{1}{V_{o}\left(  r\right)  }\int_{B_{cc}\left(  o,r\right)  }R\left(
y\right)  d\mu\left(  y\right)  =o\left(  r^{-2}\right)  ,
\]
for some point $o\in M.$ Here $R\left(  y\right)  $ is the Tanaka-Webster
scalar curvature and $V_{o}\left(  r\right)  $ is the volume of the
Carnot-Carath\'{e}odory ball $B_{cc}\left(  o,r\right)  .$

More recently, in view of Liu's approach toward Yau uniformization
Conjectures, it is important to know if there exists a nonconstant holomorphic
functions of polynomial growth in a complete noncompact $m$-dimensional
K\"{a}hler manifold $M$ of nonnegative holomorphic bisectional curvature and
positive at one point $p$. In fact, it is due to \cite{n1}, \cite{liu5}, and
\cite{liu2} that $\mathcal{O}_{p}\left(  M\right)  \neq%
%TCIMACRO{\U{2102} }%
%BeginExpansion
\mathbb{C}
%EndExpansion
$ if and only if $M$ is of maximal volume growth. Therefore one wishes to work
on Conjecture \ref{conj4} and Conjecture \ref{conj5}, it is important to know
when \
\[
\mathcal{O}_{p}^{CR}\left(  M\right)  \neq%
%TCIMACRO{\U{2102} }%
%BeginExpansion
\mathbb{C}
%EndExpansion
.
\]
In this paper, we deal with this issue as the following :

\begin{theorem}
\label{T1} There exists a nonconstant CR holomorphic function of polynomial
growth in a complete noncompact Sasakian $(2n+1)$-manifold of nonnegative
pseudohermitian bisectional curvature with the CR maximal volume growth
property%
\begin{equation}
\lim_{r\rightarrow+\infty}\frac{Vol\left(  B_{cc}\left(  p,r\right)  \right)
}{r^{2n+2}}\geq\alpha, \label{2020AAA}%
\end{equation}
for a fixed point $p$ and some positive constant $\alpha$. Here $B_{cc}\left(
p,r\right)  $ is the Carnot-Carath\'{e}odory ball in a Sasakian $(2n+1)$-manifold.
\end{theorem}

\begin{remark}
1. Note that it is due to M. Gromov (\cite{g}) that the power of the
Carnot-Carath\'{e}odory distance $r_{cc}\left(  x\right)  $ in (\ref{2020AAA})
is $2n+2$ which is crucial so that we have the CR analogue of tangent cone
property as in (\ref{2020BBB}) and Theorem \ref{t51}.

2. It is also known that any positive pseudoharmonic function in a complete
noncompact Sasakian $(2n+1)$-manifold of nonnegative pseudohermitian Ricci
curvature tensors must be constant (\cite{cklt}).
\end{remark}

Our methods and the rest of the paper is organized as follows. In section $2$,
we introduce some basic materials in a pseudohermitian $(2n+1)$-manifold and
state the CR analogue of H\"{o}rmander L$^{2}$-estimate, maximum principle for
the CR heat equation, and the CR three-circle theorem (\cite{chl2}). we will
give the detail proofs in the appendices. In section $3,$ we develope the
Cheeger-Colding theory and the metric cone structure at infinity (Theorem
\ref{t51}). In section $4$, we obtain a result which controls the size of a CR
holomorpic chart (Theorem \ref{t2}) when the manifold is Gromov-Hausdorff
close to an Euclidean ball with\ respect to the Webster (adapted) metric.
Finally in section $5$, based on the Cheeger-Colding theory, CR heat flow
technique and CR Hormander $L^{2}$-estimate, we construct CR-holomorphic
functions with controlled growth in a sequence of exhaustion domains on $M$ by
the tangent cone at infinity. On the other hand, the CR three circle Theorem
\ref{TA} ensures that we can take subsequence to obtain a nonconstant CR
holomorphic function of polynomial growth. In appendices, we derive the CR
analogue of H\"{o}rmander L$^{2}$-estimate (Proposition \ref{P21}) for
$\overline{\partial}_{B}$ in a Sasakian manifold. We also justify the weighted
basic function property (\ref{2020}) by deriving the maximum principle of CR
heat equation (Theorem \ref{t1}).

\section{Preliminaries}

We first introduce some basic materials in a pseudohermitian $(2n+1)$-manifold
(see \cite{l}). Let $(M,\xi)$ be a $(2n+1)$-dimensional, orientable, contact
manifold with contact structure $\xi$. A CR structure compatible with $\xi$ is
an endomorphism $J:\xi\rightarrow\xi$ such that $J^{2}=-1$. We also assume
that $J$ satisfies the following integrability condition: If $X$ and $Y$ are
in $\xi$, then so are $[JX,Y]+[X,JY]$ and $J([JX,Y]+[X,JY])=[JX,JY]-[X,Y]$.
Let $\left\{  \mathbf{T},Z_{\alpha},Z_{\bar{\alpha}}\right\}  $ be a frame of
$TM\otimes\mathbb{C}$, where $Z_{\alpha}$ is any local frame of $T^{1,0}%
(M),\ Z_{\bar{\alpha}}=\overline{Z_{\alpha}}\in T^{0,1}(M)$ and $\mathbf{T}$
is the characteristic vector field. Then $\left\{  \theta,\theta^{\alpha
},\theta^{\bar{\alpha}}\right\}  $, which is the coframe dual to $\left\{
\mathbf{T},Z_{\alpha},Z_{\bar{\alpha}}\right\}  $, satisfies
\begin{equation}
d\theta=ih_{\alpha\overline{\beta}}\theta^{\alpha}\wedge\theta^{\overline
{\beta}} \label{72}%
\end{equation}
for some positive definite hermitian matrix of functions $(h_{\alpha\bar
{\beta}})$. \ If we have this contact structure, we also call such $M$ a
\textbf{strictly pseudoconvex} CR $(2n+1)$-manifold. The Levi form
$\left\langle \ ,\ \right\rangle _{L_{\theta}}$ is the Hermitian form on
$T^{1,0}(M)$ defined by%
\[
\left\langle Z,W\right\rangle _{L_{\theta}}=-i\left\langle d\theta
,Z\wedge\overline{W}\right\rangle .
\]
We can extend $\left\langle \ ,\ \right\rangle _{L_{\theta}}$ to $T^{1,0}(M)$
by defining $\left\langle \overline{Z},\overline{W}\right\rangle _{L_{\theta}%
}=\overline{\left\langle Z,W\right\rangle }_{L_{\theta}}$ for all $Z,W\in
T^{1,0}(M)$. The Levi form induces naturally a Hermitian form on the dual
bundle of $T^{1,0}(M)$, denoted by $\left\langle \ ,\ \right\rangle
_{L_{\theta}^{\ast}}$, and hence on all the induced tensor bundles.
Integrating the Hermitian form (when acting on sections) over $M$ with respect
to the volume form
\[
d\mu=\theta\wedge(d\theta)^{n},
\]
we get an inner product on the space of sections of each tensor bundle.

The pseudohermitian connection of $(J,\theta)$ is the connection $\nabla$ on
$TM\otimes\mathbb{C}$ (and extended to tensors) given in terms of a local
frame $Z_{\alpha}\in T^{1,0}(M)$ by%

\[
\nabla Z_{\alpha}=\omega_{\alpha}{}^{\beta}\otimes Z_{\beta},\quad\nabla
Z_{\bar{\alpha}}=\omega_{\bar{\alpha}}{}^{\bar{\beta}}\otimes Z_{\bar{\beta}%
},\quad\nabla\mathbf{T}=0,
\]
where $\omega_{\alpha}{}^{\beta}$ are the $1$-forms uniquely determined by the
following equations:%

\[%
\begin{split}
d\theta^{\beta}  &  =\theta^{\alpha}\wedge\omega_{\alpha}{}^{\beta}%
+\theta\wedge\tau^{\beta},\\
0  &  =\tau_{\alpha}\wedge\theta^{\alpha},\\
0  &  =\omega_{\alpha}{}^{\beta}+\omega_{\bar{\beta}}{}^{\bar{\alpha}},
\end{split}
\]
We can write (by Cartan lemma) $\tau_{\alpha}=A_{\alpha\gamma}\theta^{\gamma}$
with $A_{\alpha\gamma}=A_{\gamma\alpha}$. The curvature of Tanaka-Webster
connection, expressed in terms of the coframe $\{ \theta=\theta^{0}%
,\theta^{\alpha},\theta^{\bar{\alpha}}\}$, is
\[%
\begin{split}
\Pi_{\beta}{}^{\alpha}  &  =\overline{\Pi_{\bar{\beta}}{}^{\bar{\alpha}}%
}=d\omega_{\beta}{}^{\alpha}-\omega_{\beta}{}^{\gamma}\wedge\omega_{\gamma}%
{}^{\alpha},\\
\Pi_{0}{}^{\alpha}  &  =\Pi_{\alpha}{}^{0}=\Pi_{0}{}^{\bar{\beta}}=\Pi
_{\bar{\beta}}{}^{0}=\Pi_{0}{}^{0}=0.
\end{split}
\]
Webster showed that $\Pi_{\beta}{}^{\alpha}$ can be written
\[
\Pi_{\beta}{}^{\alpha}=R_{\beta}{}^{\alpha}{}_{\rho\bar{\sigma}}\theta^{\rho
}\wedge\theta^{\bar{\sigma}}+W_{\beta}{}^{\alpha}{}_{\rho}\theta^{\rho}%
\wedge\theta-W^{\alpha}{}_{\beta\bar{\rho}}\theta^{\bar{\rho}}\wedge
\theta+i\theta_{\beta}\wedge\tau^{\alpha}-i\tau_{\beta}\wedge\theta^{\alpha}%
\]
where the coefficients satisfy
\[
R_{\beta\bar{\alpha}\rho\bar{\sigma}}=\overline{R_{\alpha\bar{\beta}\sigma
\bar{\rho}}}=R_{\bar{\alpha}\beta\bar{\sigma}\rho}=R_{\rho\bar{\alpha}%
\beta\bar{\sigma}},\ \ \ W_{\beta\bar{\alpha}\gamma}=W_{\gamma\bar{\alpha
}\beta}.
\]
Here $R_{\gamma}{}^{\delta}{}_{\alpha\bar{\beta}}$ is the pseudohermitian
curvature tensor, $R_{\alpha\bar{\beta}}=R_{\gamma}{}^{\gamma}{}_{\alpha
\bar{\beta}}$ is the pseudohermitian Ricci curvature tensor and $A_{\alpha
\beta}$\ is the pseudohermitian torsion. Furthermore, we define the
pseudohermitian bisectional curvature tensor
\[
R_{\alpha\bar{\alpha}\beta\overline{\beta}}(X,Y):=R_{\alpha\bar{\alpha}%
\beta\overline{\beta}}X_{\alpha}X_{\overline{\alpha}}Y_{\beta}Y_{\bar{\beta}}%
\]
and the pseudohermitian torsion tensor \
\[
Tor(X,Y):=i(A_{\overline{\alpha}\bar{\rho}}X^{\overline{\rho}}Y^{\overline
{\alpha}}-A_{\alpha\rho}X^{\rho}Y^{\alpha})
\]
for any $X=X^{\alpha}Z_{\alpha},\ Y=Y^{\alpha}Z_{\alpha}$ in $T^{1,0}(M).$

We will denote components of covariant derivatives with indices preceded by
comma; thus write $A_{\alpha\beta,\gamma}$. The indices $\{0,\alpha
,\bar{\alpha}\}$ indicate derivatives with respect to $\{T,Z_{\alpha}%
,Z_{\bar{\alpha}}\}$. For derivatives of a scalar function, we will often omit
the comma, for instance, $u_{\alpha}=Z_{\alpha}u,\ u_{\alpha\bar{\beta}%
}=Z_{\bar{\beta}}Z_{\alpha}u-\omega_{\alpha}{}^{\gamma}(Z_{\bar{\beta}%
})Z_{\gamma}u.$ For a smooth real-valued function $u$, the subgradient
$\nabla_{b}$ is defined by $\nabla_{b}u\in\xi$ and $\left\langle Z,\nabla
_{b}u\right\rangle _{L_{\theta}}=du(Z)$ for all vector fields $Z$ tangent to
contact plane. Locally $\nabla_{b}u=\sum_{\alpha}u_{\bar{\alpha}}Z_{\alpha
}+u_{\alpha}Z_{\bar{\alpha}}$. We also denote $u_{0}=\boldsymbol{T}u$. We can
use the connection to define the subhessian as the complex linear map
$(\nabla^{H})^{2}u:T^{1,0}(M)\oplus T^{0,1}(M)\rightarrow T^{1,0}(M)\oplus
T^{0,1}(M)$ by
\[
(\nabla^{H})^{2}u(Z)=\nabla_{Z}\nabla_{b}u.
\]
In particular%

\[%
\begin{array}
[c]{c}%
|\nabla_{b}u|^{2}=2\sum_{\alpha}u_{\alpha}u_{\overline{\alpha}},\quad
|\nabla_{b}^{2}u|^{2}=2\sum_{\alpha,\beta}(u_{\alpha\beta}u_{\overline{\alpha
}\overline{\beta}}+u_{\alpha\overline{\beta}}u_{\overline{\alpha}\beta}).
\end{array}
\]
Also
\[%
\begin{array}
[c]{c}%
\Delta_{b}u=Tr\left(  (\nabla^{H})^{2}u\right)  =\sum_{\alpha}(u_{\alpha
\bar{\alpha}}+u_{\bar{\alpha}\alpha}).
\end{array}
\]

Finally, we state the CR analogue of H\"{o}rmander L$^{2}$-estimate, maximum
principle for the CR heat equation, and the CR three-circle theorem
(\cite{chl2}) which are key results for the proof of main results in this
paper. For completeness, w\textbf{e will give the detail proofs in the
appendices}.

\begin{proposition}
\label{P21} (\textbf{CR H\"{o}rmander L}$^{2}$\textbf{-estimate}) Let
$(M,T^{1,0}(M),\theta)$ be a connected but not necessarily complete Sasakian
manifold of nonnegative pseudohermitian Ricci tensor. Assume that $M$ is Stein
in the sense that, a smooth real basic function $\psi$ on $M$ \ with
$i\partial_{B}\overline{\partial}_{B}\varphi>0$ such that the subsets $\{x\in
M:$ $\psi(x)\leq c\}$ are compact in $M$ for every real number $c$. Let
$\varphi$ be a smooth weighted basic function on $M$ with%
\begin{equation}
i\partial_{B}\overline{\partial}_{B}\varphi\geq cd\theta\label{2020}%
\end{equation}
for some positive function $c$ on $M$. Let $g$ be a smooth basic $(0,1)$-form
satisfying
\[
\overline{\partial}_{B}g=0
\]
with
\[
\int_{M}\frac{|g|^{2}}{c}e^{-\varphi}d\mu<\infty.
\]
Then there exists a smooth basic function $f$ on $M$ with
\[
\overline{\partial}_{B}f=g
\]
with
\[
\int_{M}|f|^{2}e^{-\varphi}d\mu\leq\int_{M}\frac{|g|^{2}}{c}e^{-\varphi}d\mu.
\]

\end{proposition}

By modifying the methods of \cite{nt1}, \cite{liu2} and \cite{cchl}, we should
justify the property (\ref{2020}) via the CR heat equation estimate (Theorem
\ref{t1}). Then we are able to apply the following Theorem \ref{t1} to
construct the weighted basic function as in section $4$ and section $5$.

\begin{theorem}
\label{t1} Let $\left(  M,J,\theta\right)  $ be a complete noncompact Sasakian
$\left(  2n+1\right)  $-manifold of nonnegative pseudohermitian bisectional
curvature. We consider, for $u\in C_{c}^{\infty}\left(  M\right)  $ with
$u_{0}=0,$
\[
v\left(  x,t\right)  =%
%TCIMACRO{\dint \limits_{M}}%
%BeginExpansion
{\displaystyle\int\limits_{M}}
%EndExpansion
H\left(  x,y,t\right)  u\left(  y\right)  d\mu\left(  y\right)
\]
on $M\times\left[  0,+\infty\right)  $. Let
\[
\eta_{\alpha\overline{\beta}}\left(  x,t\right)  =v_{\alpha\overline{\beta}%
}\left(  x,t\right)
\]
and $\lambda\left(  x\right)  $ be the bottom spectrum of $\eta_{\alpha
\overline{\beta}}\left(  x,0\right)  $ with%
\[
\lambda\left(  x,t\right)  =%
%TCIMACRO{\dint \limits_{M}}%
%BeginExpansion
{\displaystyle\int\limits_{M}}
%EndExpansion
H\left(  x,y,t\right)  \lambda\left(  y\right)  d\mu\left(  y\right)  .
\]
Then $\eta_{\alpha\overline{\beta}}$ is the basic $(1,1)$-tensor and
\[
\eta_{\alpha\overline{\beta}}\left(  x,t\right)  -\lambda\left(  x,t\right)
h_{\alpha\overline{\beta}}%
\]
is a nonnegative $\left(  1,1\right)  $-tensor on $M\times\left[  0,T\right]
$ for any $T>0.$
\end{theorem}

Furthermore, we also need the CR three circle theorem (\cite[Theorem
1.1.]{chl2}) to ensure that we can take subsequence to obtain a nonconstant CR
holomorphic function of polynomial growth.

\begin{theorem}
\label{TA} (\textbf{CR three circle theorem}) Let $\left(  M,J,\theta\right)
$ be a complete noncompact Sasakian $\left(  2n+1\right)  $-manifold of
nonnegative pseudohermitian sectional curvature. Let $M_{f}\left(  r\right)
=\underset{x\in B_{cc}\left(  p,r\right)  }{\sup}\left\vert f\left(  x\right)
\right\vert $ and $B_{cc}\left(  p,r\right)  $ be the Carnot-Carath\'{e}odory
ball centered at $p$ with radius $r$. Then the CR three-circle theorem holds
on $M$ in the sense that for any point $p\in M$, any positive number $R>0$,
and any function $f\in\mathcal{O}^{CR}(B_{cc}\left(  p,R\right)  )$ on the
ball $B_{cc}\left(  p,R\right)  $, $\log M_{f}\left(  r\right)  $ is convex
with respect to $\log r$ for $0<r<R$. Moreover, we have
\begin{equation}
\frac{M_{f}\left(  kr\right)  }{M_{f}\left(  r\right)  } \label{314}%
\end{equation}
is increasing with respect to $r$ for any positive number $k\geq1$.
\end{theorem}

\section{Tangent Cone at Infinity}

In this section, we will develope the metric cone structure at infinity
(Theorem \ref{t51}) for a sequence of Riemannian manifolds $(M_{i}^{m}%
,p_{i},g_{i})$ so that
\[
\mathrm{Ric}_{M_{i}}\geq-\delta_{i}^{2}%
\]
as $\delta_{i}\rightarrow0.$ Then, in the next section, we are able to handle
this situation when $\left(  M,J,\theta\right)  $ is a complete noncompact
Sasakian $\left(  2n+1\right)  $-manifold of nonnegative pseudohermitian
\ Ricci curvature with a family of the Webster metric $g_{\delta_{i}}:=g_{i}$
as in (\ref{2020A}).

Let $\displaystyle\mathcal{V}_{r}^{\kappa}(p)=\frac{vol(B_{r}(p))}
{vol(B_{r}({S}_{k}^{n}))}$, where ${{S}}_{k}^{n}$ is the
$n$-dimensional space form of constant curvature $\kappa$. For simplicity, we
denote $\mathcal{V}_{r}^{0}(p)$ by $\mathcal{V}_{r}(p)$. We have the following
easy lemma.

\begin{lemma}
\label{l:vol.comp} Let $(M^{m},g)$ be a Riemannian manifold. Then for
$\kappa\geq0$,

(i)
\[
1\leq\frac{\mathcal{V}_{r}(p)}{\mathcal{V}_{r}^{-\kappa}(p)}\leq e^{c(m)\kappa
r}.
\]
(ii) If $\ \mathrm{Ric}_{M}\geq-(m-1)\kappa^{2}$ and\textrm{ } $R>r>0,$, then%
\[
\mathcal{V}_{R}(p)\leq e^{c(m)\kappa R}\mathcal{V}_{r}(p).
\]

\end{lemma}

\begin{proof}
It's easy to see that (i) holds when $\kappa=1$, that is, $\displaystyle1\leq
\frac{\mathcal{V}_{r}(p)}{\mathcal{V}_{r}^{-1}(p)}\leq e^{c(m)r}$. Because
$\mathcal{V}_{r}^{-\kappa}(p)=\kappa^{m}\cdot\mathcal{V}_{\kappa r}^{-1}(p)$,
we have
\[
\frac{\mathcal{V}_{r}(p)}{\mathcal{V}_{r}^{-\kappa}(p)}=\frac{\mathcal{V}%
_{r}(p)}{\mathcal{V}_{\kappa r}(p)}\cdot\frac{\mathcal{V}_{\kappa r}%
(p)}{\kappa^{n}\mathcal{V}_{\kappa r}^{-1}(p)}=\frac{\mathcal{V}_{\kappa
r}(p)}{\mathcal{V}_{\kappa r}^{-1}(p)}.
\]
This implies (i) for all $\kappa>0$. Consequently, if $\mathrm{Ric}_{M}%
\geq-(m-1)\kappa^{2}$, by Bishop-Gromov volume comparison, we have
\[
\mathcal{V}_{r}(p)\geq\mathcal{V}_{r}^{-\kappa}(p)\geq\mathcal{V}_{R}%
^{-\kappa}(p)\geq e^{-c(m)\kappa R}\mathcal{V}_{R}(p).
\]

\end{proof}

\begin{lemma}
[Cone structure at infinity]\label{l:tcone.inf} Let $v,\rho>0$ be positive
constants, $\delta_{i}\rightarrow0$, and $(M_{i}^{m},p_{i},g_{i})$ be a
sequence of Riemannian manifolds so that
\[
\mathrm{Ric}_{M_{i}}\geq-\delta_{i}^{2}%
\]
and
\[
\mathcal{V}_{R}(p_{i})\geq v\delta_{i}R>0
\]
for all $R\geq\rho$. Passing to a subsequence, there exist subsequences
$\epsilon_{i}\rightarrow0$ and $R_{i}\rightarrow\infty$, so that the following
hold for $(B_{i}, \tilde p_{i},\tilde g_{i})=(B_{2R_{i}}(p_{i}),p_{i}%
,R_{i}^{-2}g_{i})$,

\begin{enumerate}
\renewcommand{\labelenumi}{(\roman{enumi})} \setlength{\itemsep}{1pt}

\item $\mathrm{Ric}_{B_{i}}\ge-\epsilon_{i}^{2}$ and $vol(B_{2}(\tilde
p_{i}))\ge c(m,v)>0$.

\item There exists a metric space $Z$ so that $(\bar B_{i},\tilde
p_{i})\overset{d_{GH}}{\longrightarrow}(\bar B_{2}(p^{*}),p^{*})$, where
$p^{*}\in C(Z)$ is a cone point.
\end{enumerate}
\end{lemma}

\begin{proof}
For simplicity, we denote $(B_{r}(p),p,R^{-1}d)$ by $R^{-1}B_{r}(p)$. Not
losing generality, assume $\delta_{i}^{-1}\geq\rho$ for all $i$. Let $\rho
_{i}=\delta_{i}^{-1}\geq\rho$ and consider $(A_{i},p_{i}^{\prime})=(\rho
_{i}^{-1}B_{2\rho_{i}}(p_{i}),p_{i})$. Then for $\rho_{i}\geq\rho$, we have

\begin{enumerate}
\renewcommand{\labelenumi}{(\arabic{enumi})} \setlength{\itemsep}{1pt}

\item $\mathrm{Ric}_{A_{i}}\ge-\delta_{i}^{2}\rho_{i}^{2}=-1$,

\item $vol(B_{2}(p_{i}^{\prime}))=\rho_{i}^{-m}\cdot vol(B_{2\rho_{i}}%
(p_{i}))=c(m)\mathcal{V}_{\rho_{i}}(p_{i})\ge c(m) v\delta_{i}\rho_{i}=c(m) v$.
\end{enumerate}

%\[%
%\begin{array}
%[c]{cl}%
%1. & \mathrm{Ric}_{A_{i}}\geq-\delta_{i}^{2}\rho_{i}^{2}=-1,\\
%2. & vol(B_{2}(p_{i}^{\prime}))=vol(\rho_{i}^{-1}B_{2\rho_{i}}(p_{i}%
%))=\rho_{i}^{-m}\cdot vol(B_{2\rho_{i}}(p_{i}))=2^{m}\mathcal{V}_{\rho_{i}%
%}(p_{i})\geq2^{m}v\delta_{i}\rho_{i}=2^{m}v.
%\end{array}
%\]

Passing to a subsequence, we have $\displaystyle\lim_{i\rightarrow\infty}%
(\bar{A}_{i},p_{i}^{\prime})=(\bar{B}_{2}(p_{\infty}),p_{\infty})$. Because
the sequence is non-collapsed convergent, the tangent cone of $\bar{B}%
_{2}(p_{\infty})$ at $p_{\infty}$ are all $m$-dimensional metric cones. In
particular, there exist $r_{j}\rightarrow0$ and a metric cone $C(Z)$, so that
for any $\epsilon>0$, there exists $N>0$ such that
\begin{equation}
d_{GH}(r_{j}^{-1}B_{2r_{j}}(p_{\infty}),B_{2}(p^{\ast}))<\epsilon\label{1a}%
\end{equation}
for all $j>N$, where $p^{\ast}\in C(Z)$ is the cone point. Now fix $j>0$ and
the corresponding $r_{j}$, because $\displaystyle(\bar{B}_{2}(p_{\infty
}),p_{\infty})=\lim_{i\rightarrow\infty}(\bar{A}_{i},p_{i}^{\prime})$, there
exists $i_{j}>0$ so that
\begin{equation}
d_{GH}(B_{2r_{j}}(p_{i}^{\prime}),B_{2r_{j}}(p_{\infty}))<r_{j}\epsilon
\label{1b}%
\end{equation}
for all $i\geq i_{j}$. Moreover, we can choose $i_{j}$ large enough such that
$\rho_{i_{j}}>r_{j}^{-2}$. Combining (\ref{1a}) and (\ref{1b}), we get
\[
d_{GH}(r_{j}^{-1}B_{2r_{j}}(p_{i_{j}}^{\prime}),B_{2}(p^{\ast}))<2\epsilon.
\]
On each $(M_{i_{j}},p_{i_{j}})$, since $B_{2r_{j}}(p_{i_{j}}^{\prime}%
)=\rho_{i_{j}}^{-1}B_{2\rho_{i_{j}}r_{j}}(p_{i_{j}})$, we get
\[
d_{GH}(\rho_{i_{j}}^{-1}r_{j}^{-1}B_{2\rho_{i_{j}}r_{j}}(p_{i_{j}}%
),B_{2}(p^{\ast}))<2\epsilon.
\]
Let $R_{j}=\rho_{i_{j}}r_{j}$, $j=1,2,\dots$. It's clear that $R_{j}%
>r_{j}^{-1}\rightarrow\infty$ and (ii) is satisfied. It remains to verify (i)
for $(B_{j}, \tilde p_{j})=R_{i_{j}}^{-1}B_{2R_{i_{j}}}(p_{i_{j}})$. First,
$\mathrm{Ric}_{B_{j}}\geq-\delta_{i_{j}}^{2}R_{j}^{2}=-\delta_{i_{j}}^{2}%
\rho_{i_{j}}^{2}r_{j}^{2}=-r_{j}^{2}\rightarrow0$. Because $R_{j}<\rho_{i_{j}%
}$, by Lemma \ref{l:vol.comp} (ii), we have
\[
vol(B_{2}(\tilde p_{j}))=R_{j}^{-m}vol(B_{2R_{j}}(p_{i_{j}}))=c(m)\mathcal{V}%
_{2R_{j}}(p_{i_{j}})\geq c(m)e^{-c(m)\delta_{i_{j}}\rho_{i_{j}}}%
\mathcal{V}_{2\rho_{i_{j}}}(p_{i_{j}})\geq c(m,v).
\]

\end{proof}

\begin{theorem}
\label{t51} \textbf{(Radii of almost Euclidean balls)} Let $v,\rho>0$ be
constants and $\delta_{i}\rightarrow0$. Let $(M_{i}^{m},p_{i},g_{i})$ be a
sequence of manifolds so that
\begin{equation}
\mathrm{Ric}_{M_{i}}\geq-\delta_{i}^{2} \label{2020DDD}%
\end{equation}
and
\begin{equation}
\mathcal{V}_{R}(p_{i})\geq v\delta_{i}R>0 \label{2020CCC}%
\end{equation}
for all $R\geq\rho$. Passing to a subsequence, there exist $r_{i}%
\rightarrow\infty$, so that the following hold. For any $\epsilon>0$, there is
$N>0$ so that for any $i>N$, there exist $y_{i},z_{i}\in M_{i}$ so that
\begin{align}
&  \frac{1}{2}c(m,v,\epsilon)r_{i}<d(p_{i},y_{i}),\,d(p_{i},z_{i}%
)<2c(m,v,\epsilon)r_{i},\\
&  d(y_{i},z_{i})>c(m,v,\epsilon)r_{i},\\
&  d_{GH}(B_{r_{i}/\epsilon}(y_{i}),B_{r_{i}/\epsilon}(0,%
%TCIMACRO{\U{211d} }%
%BeginExpansion
\mathbb{R}
%EndExpansion
^{m}))<\epsilon r_{i}\text{\quad\textrm{and} \quad}d_{GH}(B_{r_{i}/\epsilon
}(z_{i}),B_{r_{i}/\epsilon}(0,%
%TCIMACRO{\U{211d} }%
%BeginExpansion
\mathbb{R}
%EndExpansion
^{m}))<\epsilon r_{i}.
\end{align}

\end{theorem}

\begin{proof}
We adapt the notions from Lemma \ref{l:tcone.inf}. Let $M_{i}$, $\epsilon_{i}$
and $R_{i}$ be the sequences and $C(Z)$ be the metric cone constructed in
Lemma \ref{l:tcone.inf}. As defined in \cite{cjn}, let $\mathcal{S}%
_{\epsilon,r}^{k}$ be the $(k,\epsilon,r)$-singular sets. By \cite[Theorem
1.7.]{cjn} we have $\mathcal{H}^{m}(\mathcal{S}_{\epsilon,r}^{m-1}(C(Z))\cap
B_{3/2}(p^{\ast}))<c(m,\epsilon,v)r$. Thus by the metric cone structure, we
have $\mathcal{H}^{m-1}(\mathcal{S}_{\epsilon,r}^{m-1}(C(Z))\cap\partial
B_{1}(p^{\ast}))<c(m,\epsilon,v)r$. Choosing $r=r_{0}=r_{0}(m,\epsilon,v)>0$
sufficiently small, we can find $y_{0},z_{0}\in\partial B_{1}(p^{\ast
})\setminus\mathcal{S}_{\epsilon,r_{0}}^{m-1}(C(Z))$ with $d(y_{0}%
,z_{0})>c(m,\epsilon, v)$. In summary, we find two points $y_{0},z_{0}\in
B_{1}(p^{\ast})$, so that
\begin{align}
&  d(p^{\ast},y_{0})=1\text{\quad\textsf{and} \quad}d(p^{\ast},z_{0})=1,\\
&  d(y_{0},z_{0})>c(m,\epsilon,v)>0,\\
&  d_{GH}(B_{r_{0}}(y_{0}),B_{r_{0}}(0,%
%TCIMACRO{\U{211d} }%
%BeginExpansion
\mathbb{R}
%EndExpansion
^{m}))<\epsilon{r_{0}}\text{\quad\textrm{and} \quad}d_{GH}(B_{r_{0}}%
(z_{0}),B_{r_{0}}(0,%
%TCIMACRO{\U{211d} }%
%BeginExpansion
\mathbb{R}
%EndExpansion
^{m}))<\epsilon{r_{0}}.
\end{align}
Fix $r_{0}>0$. Because $(\bar{B}_{i}, \tilde p_{i})=R_{i}^{-1}\bar{B}_{2R_{i}%
}(p_{i})\overset{d_{GH}}{\longrightarrow}(\bar{B}_{2}(p^{\ast}), p^{*})$, for
$i$ large, there are $y_{i}^{\prime},z_{i}^{\prime}\in B_{i}$ so that
\begin{align}
&  1-\epsilon<d(p_{i}^{\prime},y_{i}^{\prime})<1+\epsilon\text{\quad
\textrm{and} \quad}1-\epsilon<d(p_{i}^{\prime},z_{i}^{\prime})<1+\epsilon,\\
&  d(y_{i}^{\prime},z_{i}^{\prime})>c(m,\epsilon,v)>0,\\
&  d_{GH}(B_{r_{0}}(y_{i}^{\prime}),B_{r_{0}}(y_{0}))<\epsilon r_{0}%
\text{\quad\textrm{and} \quad}d_{GH}(B_{r_{0}}(z_{i}^{\prime}),B_{r_{0}}%
(z_{0}))<\epsilon r_{0}.
\end{align}
Therefore,
\[
d_{GH}(B_{r_{0}}(y_{i}^{\prime}),B_{r_{0}}(0,%
%TCIMACRO{\U{211d} }%
%BeginExpansion
\mathbb{R}
%EndExpansion
^{m}))<2\epsilon{r_{0}}\text{\quad\textrm{and} \quad}d_{GH}(B_{r_{0}}%
(z_{i}^{\prime}),B_{r_{0}}(0,%
%TCIMACRO{\U{211d} }%
%BeginExpansion
\mathbb{R}
%EndExpansion
^{m}))<2\epsilon{r_{0}}.
\]
Now rescalling $M_{i}$ by $R_{i}$, we find $y_{i},z_{i}\in M_{i}$, so that
\begin{align}
&  (1-\epsilon)R_{i}<d(p_{i},y_{i})<(1+\epsilon)R_{i}\text{\quad\textrm{and}
\quad}(1-\epsilon)R_{i}<d(p_{i},z_{i})<(1+\epsilon)R_{i},\\
&  d(y_{i},z_{i})>c(m,\epsilon,v)R_{i},\\
&  d_{GH}(B_{r_{0}R_{i}}(y_{i}),B_{r_{0}R_{i}}(0,%
%TCIMACRO{\U{211d} }%
%BeginExpansion
\mathbb{R}
%EndExpansion
^{m}))<2r_{0}R_{i}\epsilon\text{\quad\textrm{and} \quad}d_{GH}(B_{r_{0}R_{i}%
}(z_{i}),B_{r_{0}R_{i}}(0,%
%TCIMACRO{\U{211d} }%
%BeginExpansion
\mathbb{R}
%EndExpansion
^{m}))<2r_{0}R_{i}\epsilon.
\end{align}
Let $r_{i}=r_{0}R_{i}\sqrt{\epsilon}$, where $r_{i}\rightarrow\infty$ since
$R_{i}\rightarrow\infty$. We get the desired result:
\[
d_{GH}(B_{\frac{r_{i}}{\sqrt{\epsilon}}}(y_{i}),B_{\frac{r_{i}}{\sqrt
{\epsilon}}}(0,%
%TCIMACRO{\U{211d} }%
%BeginExpansion
\mathbb{R}
%EndExpansion
^{m}))<2r_{i}\sqrt{\epsilon}\text{\quad\textrm{and} \quad}d_{GH}%
(B_{\frac{r_{i}}{\sqrt{\epsilon}}}(z_{i}),B_{\frac{r_{i}}{\sqrt{\epsilon}}}(0,%
%TCIMACRO{\U{211d} }%
%BeginExpansion
\mathbb{R}
%EndExpansion
^{m}))<2r_{i}\sqrt{\epsilon}.
\]

\end{proof}

\section{CR-Holomorphic Charts With the Uniform Size}

In this section, by applying Proposition \ref{P21} and Theorem \ref{t1}, we
will construct a CR holomorphic chart with the uniform size which is a crucial
step to prove the existence of nonconstant CR holomorphic function of
polynomial growth.

Let $\left(  M,J,\theta\right)  $ be a complete noncompact Sasakian $\left(
2n+1\right)  $-manifold of nonnegative pseudohermitian \ Ricci curvature. We
consider a family of Webster Riemannian (adapted) metrics $g_{\lambda}$ of
$(M,J,\theta)$%
\begin{equation}
g_{\lambda}=\frac{1}{2}d\theta+\lambda^{2}\theta\otimes\theta,\text{
\ \ }\lambda>0. \label{2020A}%
\end{equation}
Since the pseudohermitian torsion is vanishing, it follows from (\cite{cc})
that
\begin{equation}
Ric(g_{\lambda})\geq-2\lambda^{2} \label{2020B}%
\end{equation}
and
\[
d\mu_{\lambda}=\frac{\lambda}{2^{n}n!}d\mu.
\]

\begin{theorem}
\label{t2} Let $\left(  M,J,\theta\right)  $ be a complete noncompact Sasakian
$\left(  2n+1\right)  $-manifold of nonnegative pseudohermitian bisectional
curvature. There are constants $\epsilon\left(  n\right)  >0,$ $\delta>0$ so
that if
\[
d_{GH}\left(  B_{g_{\lambda}}\left(  x,\frac{r}{\epsilon}\right)
,B_{g_{\lambda}}^{\mathbf{H}_{n}}\left(  (0,0),\frac{r}{\epsilon}\right)
\right)  <\epsilon r
\]
for some $0<\epsilon<\epsilon\left(  n\right)  $, where $g_{\lambda}$ is the
adapted metric with $\lambda^{2}=\frac{n\epsilon^{3}}{r^{2}},$ then there
exist a CR-holomorphic chart $\left(  w_{1},...w_{n},x^{\prime}\right)  $
containing the Carnot-Carath\'{e}odory ball $B_{cc}\left(  x,\delta r\right)
$ so that%
\begin{equation}
\left\{
\begin{array}
[c]{cl}%
\left(  1\right)  & w_{s}\left(  x\right)  =0\text{ }\mathrm{for}\text{
}\mathrm{any}\text{ }s\in I_{n};\\
\left(  2\right)  & \left\vert
%TCIMACRO{\dsum \limits_{s\in I_{n}}}%
%BeginExpansion
{\displaystyle\sum\limits_{s\in I_{n}}}
%EndExpansion
\left\vert w_{s}\left(  y\right)  \right\vert ^{2}-r_{cc}^{2}\left(  y\right)
\right\vert \leq\Phi\left(  \epsilon|n\right)  r^{2}\text{ }\mathrm{in}\text{
}B_{cc}\left(  x,\delta r\right)  ;\\
\left(  3\right)  & \left\vert \nabla w_{s}\right\vert \leq C\left(
n,\epsilon,r\right)  \text{ }\mathrm{in}\text{ }B_{cc}\left(  x,\delta
r\right)  .
\end{array}
\right.  \label{1}%
\end{equation}
Here $r_{cc}\left(  y\right)  =d_{cc}\left(  x,y\right)  .$ $\mathbf{H}_{n}$
$=$ $%
%TCIMACRO{\U{2102} }%
%BeginExpansion
\mathbb{C}
%EndExpansion
^{n}\times%
%TCIMACRO{\U{211d} }%
%BeginExpansion
\mathbb{R}
%EndExpansion
$ is an $\ (2n+1)$-dimensional Heisenberg group and $\Phi\left(
\epsilon|n\right)  $ is a nonnegative function so that $\Phi\left(
\epsilon|n\right)  \rightarrow0$ as $\epsilon\rightarrow0.$
\end{theorem}

\begin{proof}
Let $\left\{  e_{j}\right\}  _{j\in I_{2n+1}}$ be an orthonormal frame with
respect to the adapted metric $g_{\lambda}$ with $e_{2n+1}=\frac{1}{\lambda
}\mathbf{T},$ $Z_{\alpha}=\frac{1}{\sqrt{2}}\left(  e_{\alpha}-ie_{\widetilde
{\alpha}}\right)  $ for $\alpha\in I_{n}$ where $e_{\widetilde{\alpha}%
}=Je_{\alpha}$. Here $\lambda\in\left(  0,1\right)  $. May assume $r\gg1$, to
be determined, and set $R=\frac{r}{100}\gg1.$ Because
\[
d_{GH}\left(  B_{g_{\lambda}}\left(  x,\frac{r}{\epsilon}\right)
,B_{g_{\lambda}}^{%
%TCIMACRO{\U{211d} }%
%BeginExpansion
\mathbb{R}
%EndExpansion
^{2n+1}}\left(  0,\frac{r}{\epsilon}\right)  \right)  <\epsilon r,
\]
by Cheeger-Colding theory (\cite[Theorem 1.2.]{chco1} or \cite[(1.23)]{chco2},
\cite{liu4}), there are real harmonic functions $\left\{  b^{j}\right\}
_{j\in I_{2n+1}}$on $B_{g_{\lambda}}\left(  x,4r\right)  $ such that%
\begin{equation}
\left\{
\begin{array}
[c]{cl}%
\left(  1\right)  & \frac{1}{Vol(B_{g_{\lambda}}\left(  x,2r\right)  )}%
%TCIMACRO{\dint \limits_{B_{g_{\lambda}}\left(  x,2r\right)  }}%
%BeginExpansion
{\displaystyle\int\limits_{B_{g_{\lambda}}\left(  x,2r\right)  }}
%EndExpansion
\left(
%TCIMACRO{\dsum \limits_{j\in I_{2n+1}}}%
%BeginExpansion
{\displaystyle\sum\limits_{j\in I_{2n+1}}}
%EndExpansion
\left\vert Hess\left(  b^{j}\right)  \right\vert ^{2}+%
%TCIMACRO{\dsum \limits_{j,k\in I_{2n+1}}}%
%BeginExpansion
{\displaystyle\sum\limits_{j,k\in I_{2n+1}}}
%EndExpansion
\left\vert \left\langle \nabla b^{j},\nabla b^{k}\right\rangle -\delta
_{jk}\right\vert ^{2}\right)  d\mu_{\lambda}\leq\Phi\left(  \epsilon
|n,r\right)  ;\\
\left(  2\right)  & b^{j}\left(  x\right)  =0\text{ }\mathrm{for}\text{
}\mathrm{any}\text{ }j\in I_{2n+1};\\
\left(  3\right)  & \left\vert \nabla b^{j}\right\vert \leq C\left(  n\right)
\text{ }\mathrm{in}\text{ }B_{g_{\lambda}}\left(  x,2r\right)  ,\text{
}\mathrm{for}\text{ }\mathrm{any}\text{ }j\in I_{2n+1};\\
\left(  4\right)  & F\left(  y\right)  =\left(  b^{1},...,b^{2n+1}\right)
\text{ }\mathrm{is}\text{ }\mathrm{a}\text{ }\Phi\left(  \epsilon|n\right)
r\mathrm{-GHA}\text{ }\mathrm{from}\text{ }B_{g_{\lambda}}\left(  x,2r\right)
\text{ }\mathrm{to}\text{ }B_{g_{\lambda}}^{%
%TCIMACRO{\U{211d} }%
%BeginExpansion
\mathbb{R}
%EndExpansion
^{2n+1}}\left(  0,2r\right)  .
\end{array}
\right.  \label{2}%
\end{equation}
By the argument of the Gram-Schmidt orthogonalization (\cite[Lemma 9.14.]%
{cct}, \cite[(20)]{liu3}), namely, doing an orthogonal transformation, we may
assume that those real harmonic functions $\left\{  b^{j}\right\}  _{j\in
I_{n}}$ satisfy
\[
\frac{1}{Vol(B_{g_{\lambda}}\left(  x,r\right)  )}%
%TCIMACRO{\dint \limits_{B_{g_{\lambda}}\left(  x,r\right)  }}%
%BeginExpansion
{\displaystyle\int\limits_{B_{g_{\lambda}}\left(  x,r\right)  }}
%EndExpansion
\left\vert J\nabla b^{2j-1}-\nabla b^{2j}\right\vert ^{2}d\mu_{\lambda}%
\leq\Phi\left(  \epsilon|n,r\right)
\]
for any $j\in I_{n}$. Set
\[
w_{j}^{\prime}=b^{2j-1}+ib^{2j}%
\]
for $1\leq j\leq n.$ Then
\begin{equation}
\frac{1}{Vol(B_{g_{\lambda}}\left(  x,r\right)  )}%
%TCIMACRO{\dint \limits_{B_{g_{\lambda}}\left(  x,r\right)  }}%
%BeginExpansion
{\displaystyle\int\limits_{B_{g_{\lambda}}\left(  x,r\right)  }}
%EndExpansion
\left\vert \overline{\partial}_{b}w_{j}^{\prime}\right\vert ^{2}d\mu_{\lambda
}\leq\Phi\left(  \epsilon|n,r\right)  . \label{3}%
\end{equation}
Here we use the fact that $b^{j}$ are basic functions on $B_{g_{\lambda}%
}\left(  x,2r\right)  $
\begin{equation}
b_{0}^{j}=0 \label{4}%
\end{equation}
for any $j\in I_{2n}$ and
\begin{equation}
\frac{1}{Vol(B_{g_{\lambda}}\left(  x,r\right)  )}%
%TCIMACRO{\dint \limits_{B_{g_{\lambda}}\left(  x,r\right)  }}%
%BeginExpansion
{\displaystyle\int\limits_{B_{g_{\lambda}}\left(  x,r\right)  }}
%EndExpansion
\left\vert b_{0}^{2n+1}-1\right\vert ^{2}d\mu_{\lambda}\leq\Phi\left(
\epsilon|n,r\right)  . \label{4b}%
\end{equation}
This is easily derived from the inequality $\left(  1\right)  $ in $\left(
\ref{2}\right)  $ and the definition of the adapted metric $g_{\lambda}$.
Consider the function
\[
k\left(  y\right)  =%
%TCIMACRO{\dsum \limits_{j\in I_{2n}}}%
%BeginExpansion
{\displaystyle\sum\limits_{j\in I_{2n}}}
%EndExpansion
\left(  b^{j}\left(  y\right)  \right)  ^{2}.
\]
Then, in $B_{g_{\lambda}}\left(  x,r\right)  $
\begin{equation}
\left\{
\begin{array}
[c]{cl}%
\left(  1\right)  & \left\vert k\left(  y\right)  -r^{2}\left(  y\right)
\right\vert \leq\Phi\left(  \epsilon|n\right)  r^{2};\\
\left(  2\right)  & \left\vert \nabla k\left(  y\right)  \right\vert \leq
C\left(  n\right)  r\left(  y\right)  ;\\
\left(  3\right)  &
%TCIMACRO{\dint \limits_{B_{g_{\lambda}}\left(  x,5R\right)  }}%
%BeginExpansion
{\displaystyle\int\limits_{B_{g_{\lambda}}\left(  x,5R\right)  }}
%EndExpansion%
%TCIMACRO{\dsum \limits_{j,l\in I_{2n}}}%
%BeginExpansion
{\displaystyle\sum\limits_{j,l\in I_{2n}}}
%EndExpansion
\left\vert k_{jl}-2(g_{\lambda})_{jl}\right\vert ^{2}d\mu_{\lambda}\leq
\Phi\left(  \epsilon|n,R\right)
\end{array}
\right.  \label{5}%
\end{equation}
by $\left(  \ref{2}\right)  $.

Now you are going to construct the weighted basic function as in Proposition
\ref{P21}. \ Set $\varphi$ to be a smooth function from $\left[
0,+\infty\right)  $ to $\left[  0,+\infty\right)  $ with compact support such
that%
\[
\varphi\left(  t\right)  =\left\{
\begin{array}
[c]{cl}%
t & \mathrm{on}\text{ }\left[  0,1\right]  ,\\
0 & \mathrm{on}\text{ }\left[  2,+\infty\right)  ,
\end{array}
\right.
\]
and $\left\vert \varphi\right\vert ,\left\vert \varphi^{\prime}\right\vert
,\left\vert \varphi^{\prime\prime}\right\vert \leq C\left(  n\right)  $. Let
\[
u\left(  y\right)  =5R^{2}\varphi\left(  \frac{k\left(  y\right)  }{5R^{2}%
}\right)  ,
\]
and%
\[
v\left(  z,t\right)  =v_{t}\left(  z\right)  =%
%TCIMACRO{\dint \limits_{M}}%
%BeginExpansion
{\displaystyle\int\limits_{M}}
%EndExpansion
H\left(  z,y,t\right)  u\left(  y\right)  d\mu\left(  y\right)  ,
\]
where $H\left(  z,y,t\right)  $ and $d\mu$ denote the CR heat kernel and CR
volume element respectively. Moreover, since $k_{0}=0,$ this implies that
$v_{0}=0.$ Now we are able to justify that $v_{1}\left(  z\right)  $ serves as
a weighted basic function as below :

\textbf{Claim 4.1 :}\ $v_{\alpha\overline{\beta}}\left(  z,1\right)  \geq
c\left(  n,\epsilon,R\right)  h_{\alpha\overline{\beta}}>0$ in $B_{g_{\lambda
}}\left(  x,\frac{R}{10}\right)  $ for sufficiently large $R$ and sufficiently
small $\epsilon$.

Proof of \textbf{Claim 4.1. }: Let
\[
\Lambda\left(  z,t\right)  =%
%TCIMACRO{\dint \limits_{M}}%
%BeginExpansion
{\displaystyle\int\limits_{M}}
%EndExpansion
H\left(  z,y,t\right)  \Lambda\left(  y\right)  d\mu\left(  y\right)  =%
%TCIMACRO{\dint \limits_{M}}%
%BeginExpansion
{\displaystyle\int\limits_{M}}
%EndExpansion
H_{\lambda}\left(  z,y,\frac{t}{2}\right)  \Lambda\left(  y\right)
d\mu_{\lambda}\left(  y\right)  ,
\]
where $\Lambda\left(  y\right)  $ is the bottom spectrum of $v_{\alpha
\overline{\beta}}\left(  z,0\right)  =u_{\alpha\overline{\beta}}\left(
z\right)  $. Here for the second equality, we use the similar arguments as in
\cite[page 29]{chl1} and the fact that
\[
u_{0}=0=v_{0};\text{ \ \ }u_{0\alpha\overline{\beta}}=u_{\alpha\overline
{\beta}0}=0
\]
and
\[
\Delta_{g_{\lambda}}\Lambda=2\Delta_{b}\Lambda+\lambda^{-2}T^{2}%
\Lambda=2\Delta_{b}\Lambda.
\]
By the inequality $\left(  3\right)  $ in $\left(  \ref{5}\right)  $,
\[%
%TCIMACRO{\dint \limits_{B_{g_{\lambda}}\left(  x,5R\right)  }}%
%BeginExpansion
{\displaystyle\int\limits_{B_{g_{\lambda}}\left(  x,5R\right)  }}
%EndExpansion
\left\vert k_{\alpha\overline{\beta}}-2h_{\alpha\overline{\beta}}\right\vert
^{2}d\mu_{\lambda}\leq\Phi\left(  \epsilon|n,R\right)  .
\]
Then there exists a set $E=\left[  k_{\alpha\overline{\beta}}|_{B_{g_{\lambda
}}\left(  x,5R\right)  }\geq\frac{1}{2}h_{\alpha\overline{\beta}}\right]  $ so
that
\[
\mu_{\lambda}\left(  B_{g_{\lambda}}\left(  x,5R\right)  \backslash E\right)
\leq\Phi\left(  \epsilon|n,R\right)  .
\]
By the inequality $\left(  2\right)  $ in $\left(  \ref{5}\right)  $, we have
$k\left(  y\right)  \leq5R^{2}$ on $B_{g_{\lambda}}\left(  x,2R\right)  $.
Hence
\[
u=k,
\]
in $B_{g_{\lambda}}\left(  x,2R\right)  $. It is not difficult to observe that%
\begin{equation}%
\begin{array}
[c]{cl}
& \left(
%TCIMACRO{\dint \limits_{B_{g_{\lambda}}\left(  x,2R\right)  \backslash E}}%
%BeginExpansion
{\displaystyle\int\limits_{B_{g_{\lambda}}\left(  x,2R\right)  \backslash E}}
%EndExpansion
\left\vert \Lambda\left(  y\right)  \right\vert ^{2}d\mu_{\lambda}\left(
y\right)  \right)  ^{\frac{1}{2}}\\
\leq & \left(
%TCIMACRO{\dint \limits_{B_{g_{\lambda}}\left(  x,4R\right)  \backslash E}}%
%BeginExpansion
{\displaystyle\int\limits_{B_{g_{\lambda}}\left(  x,4R\right)  \backslash E}}
%EndExpansion%
%TCIMACRO{\dsum \limits_{\alpha,\beta}}%
%BeginExpansion
{\displaystyle\sum\limits_{\alpha,\beta}}
%EndExpansion
\left\vert k_{\alpha\overline{\beta}}\right\vert ^{2}d\mu_{\lambda}\left(
y\right)  \right)  ^{\frac{1}{2}}\\
\leq & \left\{
%TCIMACRO{\dint \limits_{B_{g_{\lambda}}\left(  x,4R\right)  \backslash E}}%
%BeginExpansion
{\displaystyle\int\limits_{B_{g_{\lambda}}\left(  x,4R\right)  \backslash E}}
%EndExpansion
\left[
%TCIMACRO{\dsum \limits_{\alpha,\beta}}%
%BeginExpansion
{\displaystyle\sum\limits_{\alpha,\beta}}
%EndExpansion
\left(  \left\vert k_{\alpha\overline{\beta}}-2h_{\alpha\overline{\beta}%
}\right\vert +\left\vert 2h_{\alpha\overline{\beta}}\right\vert \right)
\right]  ^{2}d\mu_{\lambda}\left(  y\right)  \right\}  ^{\frac{1}{2}}\\
\leq &
%TCIMACRO{\dsum \limits_{\alpha,\beta}}%
%BeginExpansion
{\displaystyle\sum\limits_{\alpha,\beta}}
%EndExpansion
\left(  \left\Vert k_{\alpha\overline{\beta}}-2h_{\alpha\overline{\beta}%
}\right\Vert _{L^{2}\left(  B_{g_{\lambda}}\left(  x,4R\right)  \backslash
E,\mu_{\lambda}\right)  }+2\left\Vert h_{\alpha\overline{\beta}}\right\Vert
_{L^{2}\left(  B_{g_{\lambda}}\left(  x,4R\right)  \backslash E,\mu_{\lambda
}\right)  }\right) \\
\leq & \Phi\left(  \epsilon|n,R\right)
\end{array}
\label{6}%
\end{equation}
and%
\[%
\begin{array}
[c]{ccl}%
\left\vert \Lambda\left(  y\right)  \right\vert  & \leq &
%TCIMACRO{\dsum \limits_{\alpha,\beta}}%
%BeginExpansion
{\displaystyle\sum\limits_{\alpha,\beta}}
%EndExpansion
\left\vert u_{\alpha\overline{\beta}}\left(  y\right)  \right\vert \\
& \leq &
%TCIMACRO{\dsum \limits_{\alpha,\beta}}%
%BeginExpansion
{\displaystyle\sum\limits_{\alpha,\beta}}
%EndExpansion
\left\vert \varphi^{\prime}k_{\alpha\overline{\beta}}+\frac{\varphi
^{\prime\prime}}{5R^{2}}k_{\alpha}k_{\overline{\beta}}\right\vert \\
& \leq &
%TCIMACRO{\dsum \limits_{\alpha,\beta}}%
%BeginExpansion
{\displaystyle\sum\limits_{\alpha,\beta}}
%EndExpansion
\left\vert \varphi^{\prime}\left(  k_{\alpha\overline{\beta}}-2h_{\alpha
\overline{\beta}}\right)  \right\vert +%
%TCIMACRO{\dsum \limits_{\alpha,\beta}}%
%BeginExpansion
{\displaystyle\sum\limits_{\alpha,\beta}}
%EndExpansion
\left\vert 2\varphi^{\prime}h_{\alpha\overline{\beta}}+\frac{\varphi
^{\prime\prime}}{5R^{2}}k_{\alpha}k_{\overline{\beta}}\right\vert \\
& \leq & C\left(  n\right)
%TCIMACRO{\dsum \limits_{\alpha,\beta}}%
%BeginExpansion
{\displaystyle\sum\limits_{\alpha,\beta}}
%EndExpansion
\left(  \left\vert \left(  k_{\alpha\overline{\beta}}-2h_{\alpha
\overline{\beta}}\right)  \right\vert +1\right)
\end{array}
\]
for $y\in B_{g_{\lambda}}\left(  x,5R\right)  $. Accordingly,
\begin{equation}%
%TCIMACRO{\dint \limits_{B_{g_{\lambda}}\left(  x,5R\right)  }}%
%BeginExpansion
{\displaystyle\int\limits_{B_{g_{\lambda}}\left(  x,5R\right)  }}
%EndExpansion
\left\vert \Lambda\left(  y\right)  \right\vert d\mu_{\lambda}\left(
y\right)  \leq C\left(  n\right)  R^{2n+1} \label{7}%
\end{equation}
by Cheeger-Colding theory. From the definition of the set $E$, we see that
\begin{equation}
\Lambda\geq\frac{1}{2} \label{8}%
\end{equation}
on $E\cap B_{g_{\lambda}}\left(  x,2R\right)  $. For any $z\in B_{g_{\lambda}%
}\left(  x,\frac{R}{10}\right)  $,
\begin{equation}%
\begin{array}
[c]{ccl}%
\Lambda\left(  z,1\right)  & = &
%TCIMACRO{\dint \limits_{M}}%
%BeginExpansion
{\displaystyle\int\limits_{M}}
%EndExpansion
H_{\lambda}\left(  z,y,\frac{1}{2}\right)  \Lambda\left(  y\right)
d\mu_{\lambda}\left(  y\right) \\
& \geq & \underset{\left(  i\right)  }{\underbrace{%
%TCIMACRO{\dint \limits_{E\cap B_{g_{\lambda}}\left(  z,1\right)  }}%
%BeginExpansion
{\displaystyle\int\limits_{E\cap B_{g_{\lambda}}\left(  z,1\right)  }}
%EndExpansion
H_{\lambda}\left(  z,y,\frac{1}{2}\right)  \Lambda\left(  y\right)
d\mu_{\lambda}\left(  y\right)  }}+\\
&  & \underset{\left(  ii\right)  }{\underbrace{%
%TCIMACRO{\dint \limits_{B_{g_{\lambda}}\left(  x,2R\right)  \backslash E}}%
%BeginExpansion
{\displaystyle\int\limits_{B_{g_{\lambda}}\left(  x,2R\right)  \backslash E}}
%EndExpansion
H_{\lambda}\left(  z,y,\frac{1}{2}\right)  \Lambda\left(  y\right)
d\mu_{\lambda}\left(  y\right)  }}+\\
&  & \underset{\left(  iii\right)  }{\underbrace{%
%TCIMACRO{\dint \limits_{B_{g_{\lambda}}\left(  x,4R\right)  \backslash
%B_{g_{\lambda}}\left(  x,2R\right)  }}%
%BeginExpansion
{\displaystyle\int\limits_{B_{g_{\lambda}}\left(  x,4R\right)  \backslash
B_{g_{\lambda}}\left(  x,2R\right)  }}
%EndExpansion
H_{\lambda}\left(  z,y,\frac{1}{2}\right)  \Lambda\left(  y\right)
d\mu_{\lambda}\left(  y\right)  }}.
\end{array}
\label{9}%
\end{equation}
Here we utilize the fact that $supp\left(  u\right)  \subseteq B_{g_{\lambda}%
}\left(  x,4R\right)  $. From heat kernel estimates and volume doubling
property, we have

$\left(  i\right)  $%
\begin{equation}%
\begin{array}
[c]{cl}
&
%TCIMACRO{\dint \limits_{E\cap B_{g_{\lambda}}\left(  z,1\right)  }}%
%BeginExpansion
{\displaystyle\int\limits_{E\cap B_{g_{\lambda}}\left(  z,1\right)  }}
%EndExpansion
H_{\lambda}\left(  z,y,\frac{1}{2}\right)  \Lambda\left(  y\right)
d\mu_{\lambda}\left(  y\right) \\
\geq & \frac{1}{2}%
%TCIMACRO{\dint \limits_{E\cap B_{g_{\lambda}}\left(  z,1\right)  }}%
%BeginExpansion
{\displaystyle\int\limits_{E\cap B_{g_{\lambda}}\left(  z,1\right)  }}
%EndExpansion
H_{\lambda}\left(  z,y,\frac{1}{2}\right)  d\mu_{\lambda}\left(  y\right) \\
\geq & c\left(  n,\epsilon,R\right)  >0
\end{array}
\label{10}%
\end{equation}
due to the expression of heat kernel $P_{2n+1}\left(  x,y,t\right)  $ of the
hyperbolic space $\mathcal{H}^{2n+1}$ that
\[
P_{2n+1}\left(  x,y,t\right)  =\left(  -\frac{1}{2\pi}\right)  ^{n}\frac
{1}{\sqrt{4\pi t}}\left(  \frac{1}{\sinh\rho}\frac{\partial}{\partial\rho
}\right)  ^{n}e^{-n^{2}t-\frac{\rho^{2}}{4t}}%
\]
where $\rho=\rho\left(  y\right)  $ denotes $d_{\mathcal{H}}\left(
x,y\right)  $.

$\left(  ii\right)  $%
\begin{equation}%
\begin{array}
[c]{cl}
& \left\vert
%TCIMACRO{\dint \limits_{B_{g_{\lambda}}\left(  x,2R\right)  \backslash E}}%
%BeginExpansion
{\displaystyle\int\limits_{B_{g_{\lambda}}\left(  x,2R\right)  \backslash E}}
%EndExpansion
H_{\lambda}\left(  z,y,\frac{1}{2}\right)  \Lambda\left(  y\right)
d\mu_{\lambda}\left(  y\right)  \right\vert \\
\leq &
%TCIMACRO{\dint \limits_{B_{g_{\lambda}}\left(  x,2R\right)  \backslash E}}%
%BeginExpansion
{\displaystyle\int\limits_{B_{g_{\lambda}}\left(  x,2R\right)  \backslash E}}
%EndExpansion
H_{\lambda}\left(  z,y,\frac{1}{2}\right)  \left\vert \Lambda\left(  y\right)
\right\vert d\mu_{\lambda}\left(  y\right) \\
\leq & C\left(  n,\epsilon,R\right)
%TCIMACRO{\dint \limits_{B_{g_{\lambda}}\left(  x,2R\right)  \backslash E}}%
%BeginExpansion
{\displaystyle\int\limits_{B_{g_{\lambda}}\left(  x,2R\right)  \backslash E}}
%EndExpansion
\left\vert \Lambda\left(  y\right)  \right\vert d\mu_{\lambda}\left(  y\right)
\\
\leq & C\left(  n,\epsilon,R\right)  \left\Vert \Lambda\right\Vert
_{L^{2}\left(  B_{g_{\lambda}}\left(  x,2R\right)  \backslash E,\mu_{\lambda
}\right)  }\mu_{\lambda}^{\frac{1}{2}}\left(  B_{g_{\lambda}}\left(
x,2R\right)  \backslash E\right) \\
\leq & \Phi\left(  \epsilon|n,R\right)
\end{array}
\label{11}%
\end{equation}
by the inequality $\left(  \ref{6}\right)  $.

$\left(  iii\right)  $%
\begin{equation}%
\begin{array}
[c]{cl}
&
%TCIMACRO{\dint \limits_{B_{g_{\lambda}}\left(  x,4R\right)  \backslash
%B_{g_{\lambda}}\left(  x,2R\right)  }}%
%BeginExpansion
{\displaystyle\int\limits_{B_{g_{\lambda}}\left(  x,4R\right)  \backslash
B_{g_{\lambda}}\left(  x,2R\right)  }}
%EndExpansion
\left\vert H_{\lambda}\left(  z,y,\frac{1}{2}\right)  \Lambda\left(  y\right)
\right\vert d\mu_{\lambda}\left(  y\right) \\
\leq & C\left(  n,\epsilon,R\right)  \exp\left(  -\frac{R^{2}}{5}\right)
R^{2n+1}\\
= & o\left(  1\right)
\end{array}
\label{12}%
\end{equation}
as $R\longrightarrow+\infty$.

Letting $R$ be sufficiently large and then $\epsilon$ be very small, $\left(
\ref{9}\right)  -\left(  \ref{12}\right)  $ give us
\[
\Lambda\left(  z,1\right)  \geq c\left(  n,\epsilon,R\right)  h_{\alpha
\overline{\beta}}>0
\]
on $B_{g_{\lambda}}\left(  x,\frac{R}{10}\right)  $. Combined this with
Theorem \ref{t1}, we complete the proof of \textbf{Claim 4.1. }The heat kernel
estimate enables us to conclude the proof easily as follows (\cite{liu2}) :

\textbf{Claim 4.2 :}\quad There exists a number $\epsilon_{0}=\epsilon
_{0}\left(  n,\epsilon,R\right)  >0$ such that
\[
4\sup_{B_{g_{\lambda}}\left(  x,\epsilon_{0}\frac{R}{20}\right)  }v_{1}%
<\inf_{\partial B_{g_{\lambda}}\left(  x,\frac{R}{20}\right)  }v_{1}%
\]
for sufficiently large $R$.

Therefore, we could choose a sufficiently large $R\left(  n\right)  >0$ such
that \textbf{Claim 4.1. and Claim 4.2.} hold and $\epsilon_{0}\frac{R}%
{20}>100$. Let $\Omega$ be a connected component of
\[
\left\{  y\in B_{g_{\lambda}}\left(  x,\frac{R}{20}\right)  |\text{ \ \ }%
v_{1}\left(  y\right)  <2\sup_{B_{g_{\lambda}}\left(  x,\epsilon_{0}\frac
{R}{20}\right)  }v_{1}\right\}
\]
containing $B_{g_{\lambda}}\left(  x,\epsilon_{0}\frac{R}{20}\right)  $. With
the help of \textbf{Claim 4.2.}, we know that $\Omega$ is relatively compact
in $B_{g_{\lambda}}\left(  x,\frac{R}{20}\right)  $ and $\Omega$ is Stein by
\textbf{Claim 4.1.} \ As in \cite{liu2}, by the CR analogue of H\"{o}rmander's
$L^{2}$-estimate for the space of basic forms (Proposition \ref{P21}), we
choose
\[
g_{s}:=\overline{\partial}_{B}w_{s}^{\prime}\text{ \ \ \textrm{and } }%
\varphi:=v_{1}%
\]
so that there exists a smooth basic function $f_{s}$ in $\Omega$ with
\begin{equation}
\overline{\partial}_{B}f_{s}=g_{s} \label{22}%
\end{equation}
such that%
\begin{equation}
\int_{\Omega}|f_{s}|^{2}e^{-\varphi}d\mu\leq\int_{\Omega}\frac{|g|^{2}}%
{c}e^{-\varphi}d\mu\leq\frac{1}{c\left(  n,\epsilon,R\right)  }\int_{\Omega
}|\overline{\partial}_{B}w_{s}^{\prime}|^{2}d\mu\leq\Phi\left(  \epsilon
|n\right)  . \label{22b}%
\end{equation}
But by Lemma \ref{l1}, it follows from (\ref{22b}) that
\[%
%TCIMACRO{\dint \limits_{\Omega}}%
%BeginExpansion
{\displaystyle\int\limits_{\Omega}}
%EndExpansion
\left\vert f_{s}\right\vert ^{2}d\mu\leq\Phi\left(  \epsilon|n\right)  .
\]
Hence, it is not difficult to deduce that
\[
w_{s}:=w_{s}^{\prime}-f_{s}%
\]
is a CR-holomorphic function by observing that
\begin{equation}
(w_{s}^{\prime})_{0}=0=(f_{s})_{0} \label{23}%
\end{equation}
due to the vanishing torsion. This implies $w_{s}$ is pseudoharmonic. By the
CR mean value inequality (\cite{ccht}) and the CR analogue of Cheng-Yau's
gradient estimate (\cite{cklt}), we obtain
\begin{equation}
\left\vert f_{s}\right\vert \leq\Phi\left(  \epsilon|n\right)  \text{\quad}
\label{21a}%
\end{equation}
and%
\begin{equation}
\left\vert \nabla f_{s}\right\vert \leq\Phi\left(  \epsilon|n\right)
\label{21b}%
\end{equation}
in $B_{g_{\lambda}}\left(  x,5\right)  $. Then, by $\left(  1\right)  $ in
$\left(  \ref{2}\right)  $, we have%
\begin{equation}%
%TCIMACRO{\dint \limits_{B_{g_{\lambda}}\left(  x,4\right)  }}%
%BeginExpansion
{\displaystyle\int\limits_{B_{g_{\lambda}}\left(  x,4\right)  }}
%EndExpansion
\left\vert \left(  w_{s}\right)  _{i}\overline{\left(  w_{t}\right)  _{j}%
}\left(  g^{\lambda}\right)  ^{i\overline{j}}-2\delta_{ij}\right\vert
d\mu_{\lambda}\leq\Phi\left(  \epsilon|n\right)  . \label{21}%
\end{equation}

\ \ By comparing (\ref{2018}), we observe that $\left\{  Z_{\alpha}%
=\frac{\partial}{\partial z^{\alpha}}+i\overline{z}^{\alpha}\frac{\partial
}{\partial t}\right\}  _{\alpha=1}^{n}$ is exactly a local frame in the
$(2n+1)$-dimensional Heisenberg group $\mathbf{H}_{n}$ $=$ $%
%TCIMACRO{\U{2102} }%
%BeginExpansion
\mathbb{C}
%EndExpansion
^{n}\times%
%TCIMACRO{\U{211d} }%
%BeginExpansion
\mathbb{R}
%EndExpansion
$ with the local coordinate $(z,t)$. Here
\[
\theta=dt+i%
%TCIMACRO{\dsum \limits_{\alpha\in I_{n}}}%
%BeginExpansion
{\displaystyle\sum\limits_{\alpha\in I_{n}}}
%EndExpansion
\left(  z^{\alpha}d\overline{z}^{\alpha}-\overline{z}^{\alpha}dz^{\alpha
}\right)
\]
is a pseudohermitian contact structure on $\mathbf{H}_{n}$ and $T=\frac
{\partial}{\partial\widetilde{t}}$. Hence from (\ref{23}) and (\ref{2018})
that
\begin{equation}
Z_{\alpha}(w_{s})=0=Z_{\alpha}(f_{s}). \label{24}%
\end{equation}
Consider the map
\[
\Psi:\underset{\left(  (w,b_{2n+1}),x\right)  \text{ \ }\mapsto\left(
(z,t),(\mathbf{0},0)\right)  }{B_{g_{\lambda}}\left(  x,\widetilde{c}\left(
n\right)  \right)  \longrightarrow B_{g_{\lambda}}^{\mathbf{H}_{n}}\left(
(\mathbf{0},0),\widetilde{c}\left(  n\right)  \right)  }.
\]
It follows from (\ref{21a}), (\ref{21b}), (\ref{21}), (\ref{24}) and
\textbf{claim 3.2} as in \cite[section $3$.]{liu2}, we have that
$\Psi|_{\widetilde{D_{\alpha}^{n}}}$ is an isomorphism
\[
\Psi|_{\widetilde{D_{\alpha}^{n}}}:\underset{\left(  w_{s},\widetilde
{x}\right)  \text{ \ }\mapsto\left(  z^{j},\mathbf{0}\right)  }{B_{g_{\lambda
}}\left(  x,\widetilde{c}\left(  n\right)  \right)  |_{\widetilde{D_{\alpha
}^{n}}}\longrightarrow B_{g_{\lambda}}^{\mathbf{H}_{n}}\left(  (\mathbf{0}%
,0),\widetilde{c}\left(  n\right)  \right)  |_{V_{\alpha}}},\text{
\ \ }(\widetilde{x},\widetilde{t})\in\widetilde{D_{\alpha}^{n}}%
\]
along the slice $\widetilde{D_{\alpha}^{n}}$ with%
\[
Z_{\alpha}(w_{s})=0\text{ \ \ \textrm{and} \ }(w_{s})_{0}=0.
\]
These together with (\ref{4b}), we have
\[
\Psi:\underset{\left(  (w,b_{2n+1}),x\right)  \text{ \ }\mapsto\left(
(z,t),(\mathbf{0},0)\right)  }{B_{g_{\lambda}}\left(  x,\widetilde{c}\left(
n\right)  \right)  \longrightarrow B_{g_{\lambda}}^{\mathbf{H}_{n}}\left(
(\mathbf{0},0),\widetilde{c}\left(  n\right)  \right)  }%
\]
is an CR-isomorphism from $B_{g_{\lambda}}\left(  x,\widetilde{c}\left(
n\right)  \right)  $ by defining along the corresponding geodesic integral
curve with respect to the Killing Reeb vector field $T=\frac{\partial
}{\partial\widetilde{t}}$
\[
\left(  b_{2n+1},x\right)  \text{ \ }\mapsto\left(  t,(\mathbf{0},0)\right)
.
\]
Finally, we can make a small perturbation such that $w_{s}(x)$ for $1\leq
s\leq n$. Then the proof of Theorem \ref{t2} is completed.
\end{proof}

\section{Nonconstant CR-Holomorphic Functions of Polynomial Growth}

Based on the Cheeger-Colding theory, CR heat flow technique and CR Hormander
$L^{2}$-estimate, we construct CR-holomorphic functions with controlled growth
in a sequence of exhaustion domains on $M$ by the tangent cone at infinity. On
the other hand, the CR three circle theorem (\cite{chl2}) ensures that we can
take subsequence to obtain a nonconstant CR holomorphic function of polynomial growth.

\begin{theorem}
\label{T61} Let $\left(  M,J,\theta\right)  $ be a complete noncompact
Sasakian $\left(  2n+1\right)  $-manifold of nonnegative pseudohermitian
bisectional curvature with the CR maximal volume growth property. Then there
exists a nonconstant CR-holomorphic function with polynomial growth on $M$.
\end{theorem}

\begin{proof}
Fix $p\in M$ as the reference point. Let $\delta_{i}\rightarrow0$ and define
$(M_{i},p_{i},g_{i})=(M,p,g^{\delta_{i}})$. Unlike the case in \cite{liu1},
the lower bound of $Ric_{M_{i}}$ is not zero. Thus we can't use
Cheeger-Colding theory directly on the tangent cones at infinity. In order to
get a metric cone by blowing down, we need a careful selection of the radii as
well as a stronger volume growth condition. Theorem \ref{t51} was developed to
handle this situation.

We first show that $M_{i}$ satisfy the conditions of Theorem \ref{t51}. Since
the pseudohermitian Ricc curvature is nonnegative, the Ricci curvature with
respect to the Webster metric
\[
g^{\delta_{i}}=h+\delta_{i}^{2}\theta\otimes\theta
\]
says that
\[
Ric_{M_{i}}\geq-2\delta_{i}^{2}.
\]
We observe that the CR
%TCIMACRO{\QTR{frametitle}{maximal volume growth condition}}%
%BeginExpansion
\title{maximal volume growth condition}%
%EndExpansion%
\[
\lim_{r\rightarrow\infty}\frac{Vol_{cc}(B_{r}(p))}{r^{2n+2}}\rightarrow v>0.
\]
That is, there exists $\rho>0$, depending only on $M$, so that $\frac
{Vol_{cc}(B_{r}(p))}{r^{2n+2}}\geq\frac{1}{2}v$ for all $r\geq\rho$. It is
important that for any $i$, we have
\begin{equation}
\frac{Vol_{M_{i}}(B_{r}(p_{i}))}{r^{2n+1}}=r\delta_{i}\cdot\frac
{Vol_{cc}(B_{r}(p))}{r^{2n+2}}\geq r\delta_{i}\cdot\frac{v}{2}.
\label{2020BBB}%
\end{equation}
for all $r\geq\rho$.

Now we have the property (\ref{2020CCC}) and then can apply Theorem \ref{t51}.
\ Passing to a subsequence of $M_{i}$, for any $\epsilon>0$, there exist
$\widetilde{r}_{i}\rightarrow\infty$ and $y_{i},z_{i}\in M_{i}$ so that
\[
\left\{
\begin{array}
[c]{l}%
\frac{1}{2}c(n,v,\epsilon)\widetilde{r}_{i}<d(p_{i},y_{i})\text{ \textrm{and
}}\,d(p_{i},z_{i})<2c(n,v,\epsilon)\widetilde{r}_{i}\\
d(y_{i},z_{i})>c(n,v,\epsilon)\widetilde{r}_{i}\\
d_{GH}(B_{\widetilde{r}_{i}/\epsilon}(y_{i}),B_{\widetilde{r}_{i}/\epsilon}(0,%
%TCIMACRO{\U{211d} }%
%BeginExpansion
\mathbb{R}
%EndExpansion
^{2n+1}))<\epsilon\widetilde{r}_{i}\text{ \textrm{and }}d_{GH}(B_{\widetilde
{r}_{i}/\epsilon}(z_{i}),B_{\widetilde{r}_{i}/\epsilon}(0,%
%TCIMACRO{\U{211d} }%
%BeginExpansion
\mathbb{R}
%EndExpansion
^{2n+1}))<\epsilon\widetilde{r}_{i}%
\end{array}
\right.  .
\]

By adapting $M_{i}$ (may assume that $y_{i}$ and $z_{i}$ both lie on $\partial
B\left(  p_{i},1\right)  $), Theorem \ref{t2} shows that there are the
CR-holomorphic charts $\left\{  w_{s}^{i},x^{i}\right\}  _{s\in I_{n}}$ and
$\left\{  v_{s}^{i},\widetilde{x}^{i}\right\}  _{s\in I_{n}}$ around $y_{i}$
and $z_{i}$ for any $i\in%
%TCIMACRO{\U{2115} }%
%BeginExpansion
\mathbb{N}
%EndExpansion
$. That is, by choosing a sufficiently small number $\delta$, we may assume
$\left\{  w_{s}^{i},x^{i}\right\}  _{s\in I_{n}}$ and $\left\{  v_{s}%
^{i},\widetilde{x}^{i}\right\}  _{s\in I_{n}}$ are the CR-holomorphic charts
in $B(y_{i},\delta)$ and $B(z_{i},\delta),$ respectively. Furthermore, we have%
\[
\left\{
\begin{array}
[c]{cl}%
\left(  1\right)  & w_{s}^{i}\left(  y_{i}\right)  =0=v_{s}^{i}\left(
z_{i}\right)  \text{ }\mathrm{for}\text{ }\mathrm{any}\text{ }s\in I_{n}.\\
\left(  2\right)  &
\begin{array}
[c]{l}%
\left\vert
%TCIMACRO{\dsum \limits_{s\in I_{n}}}%
%BeginExpansion
{\displaystyle\sum\limits_{s\in I_{n}}}
%EndExpansion
\left\vert w_{s}^{i}\left(  y\right)  \right\vert ^{2}-d_{i}\left(
y,y_{i}\right)  \right\vert \leq\Phi\left(  \epsilon|n\right)  \delta
^{2}\text{ }\mathrm{in}\text{ }B(y_{i},\delta),\\
\left\vert
%TCIMACRO{\dsum \limits_{s\in I_{n}}}%
%BeginExpansion
{\displaystyle\sum\limits_{s\in I_{n}}}
%EndExpansion
\left\vert v_{s}^{i}\left(  z\right)  \right\vert ^{2}-d_{i}\left(
z,z_{i}\right)  \right\vert \leq\Phi\left(  \epsilon|n\right)  \delta
^{2}\text{ }\mathrm{in}\text{ }B(z_{i},\delta).
\end{array}
\\
\left(  3\right)  &
\begin{array}
[c]{l}%
\left\vert \nabla w_{s}^{i}\right\vert \leq C\left(  n\right)  \text{
}\mathrm{in}\text{ }B(y_{i},\delta),\\
\left\vert \nabla v_{s}^{i}\right\vert \leq C\left(  n\right)  \text{
}\mathrm{in}\text{ }B(z_{i},\delta).
\end{array}
\end{array}
\right.
\]

Subsequently, we need to construct a weight function on $B(p_{i},R)$ where $R$
will be determined later. For the sake of Ricci curvature bounded from below.
Denote the annulus $A_{i}\left(  p_{i};\frac{1}{5R},5R\right)  $ in $M_{i}$ by
$A_{i}$. By the Cheeger-Colding theory (\cite[(4.2) and (4.3)]{liu4},
\cite[(4.43) and (4.82)]{chco1}), it ensures that there exists a smooth
function $\rho_{i}$ on $M_{i}$ such that%
\begin{equation}
\left\{
\begin{array}
[c]{l}%
%TCIMACRO{\dint \limits_{A_{i}}}%
%BeginExpansion
{\displaystyle\int\limits_{A_{i}}}
%EndExpansion
(\left\vert \nabla\rho_{i}-\frac{1}{2}\nabla r_{i}^{2}\right\vert
^{2}+\left\vert Hess\left(  \rho_{i}\right)  -g_{i}\right\vert ^{2}%
)<\Phi\left(  \frac{1}{i}|n,R\right)  ,\\
\left\vert \rho_{i}-\frac{1}{2}r_{i}^{2}\right\vert <\Phi\left(  \frac{1}%
{i}|n,R\right)  \text{ }\mathrm{in}\text{ }A_{i}.
\end{array}
\right.  \label{p1}%
\end{equation}
Here $r_{i}\left(  y\right)  =d_{i}\left(  y,p_{i}\right)  $. Besides, we have
(\cite[(4.20)-(4.23)]{chco1})
\begin{equation}
\rho_{i}=\frac{1}{2}\left(  F_{i}^{-1}\mathcal{G}_{i}\right)  ^{2}, \label{p2}%
\end{equation}
where $F_{i}$ is the Green function on the $\left(  2n+1\right)  $-dimensional
real space form of Ricci curvature $-2\delta_{i}^{2}$ and $\mathcal{G}_{i}$ is
the harmonic function on the annulus $A_{i}\left(  p_{i};\frac{1}%
{10R},10R\right)  $ with the boundary condition
\[
\mathcal{G}_{i}|_{\partial A_{i}\left(  p_{i};\frac{1}{10R},10R\right)
}=F_{i}.
\]
Note that as in \cite[Lemma 3.1.]{cklt}, there is an orthonormal frame such
that $(r_{i})_{0}=0,$ $|\nabla r_{i}|=1$ and then $(\rho_{i})_{0}=0$ by
(\ref{p1}). \ With the help of (\ref{p1}) and (\ref{p2}), the CR analogue of
Cheng-Yau's gradient estimate (\cite[(3.2)]{cklt}) gives us that, on the
annulus $A_{i}$,
\begin{equation}
\left\vert \nabla\rho_{i}\right\vert \left(  y\right)  \leq C\left(  n\right)
r_{i}\left(  y\right)  , \label{p3}%
\end{equation}
for sufficiently large $i$. Let $\overline{\varphi}:%
%TCIMACRO{\U{211d} }%
%BeginExpansion
\mathbb{R}
%EndExpansion
^{+}\longrightarrow%
%TCIMACRO{\U{211d} }%
%BeginExpansion
\mathbb{R}
%EndExpansion
^{+}$ be a smooth function with%
\[
\overline{\varphi}\left(  t\right)  =\left\{
\begin{array}
[c]{cl}%
0 & \mathrm{on}\text{ }\left[  0,1\right]  ,\\
t & \mathrm{on}\text{ }\left[  2,+\infty\right)  ,
\end{array}
\right.
\]
and
\begin{equation}
\sup_{\left[  0,2\right]  }\left\{  \left\vert \overline{\varphi}\right\vert
,\left\vert \overline{\varphi}^{\prime}\right\vert ,\left\vert \overline
{\varphi}^{\prime\prime}\right\vert \right\}  \leq C\left(  n\right)  .
\label{p6}%
\end{equation}
Set%
\[
u_{i}=\frac{1}{R^{2}}\overline{\varphi}\left(  R^{2}\rho_{i}\right)
\]
on $A_{i}$. Extend $\overline{\varphi}$ to be zero in $B\left(  p_{i},\frac
{1}{5R}\right)  $. As for the hypotheses of \textbf{Claim 5.2.}, we need the
following facts :

\textbf{Claim 5.1 :}
\begin{equation}
\left\{
\begin{array}
[c]{l}%
%TCIMACRO{\dint \limits_{B\left(  p_{i},4R\right)  }}%
%BeginExpansion
{\displaystyle\int\limits_{B\left(  p_{i},4R\right)  }}
%EndExpansion
(\left\vert \nabla u_{i}-\frac{1}{2}\nabla r_{i}^{2}\right\vert ^{2}%
+\left\vert Hess\left(  u_{i}\right)  -g_{i}\right\vert ^{2})<\Phi\left(
\frac{1}{R},\frac{1}{i}|n\right)  ,\\
\left\vert u_{i}-\frac{1}{2}r_{i}^{2}\right\vert <\Phi\left(  \frac{1}%
{R},\frac{1}{i}|n\right)  \text{ }\mathrm{in}\text{ }B\left(  p_{i},4R\right)
,
\end{array}
\right.  \label{p7}%
\end{equation}
and
\begin{equation}
\left\vert \nabla u_{i}\right\vert \leq C\left(  n\right)  r_{i} \label{p8}%
\end{equation}
in $B\left(  p_{i},4R\right)  $ for sufficiently large $i$.

The proof of \textbf{Claim 5.1.} could be described briefly as follows. We
observe that%
\begin{equation}
\nabla u_{i}=\overline{\varphi}^{\prime}\left(  R^{2}\rho_{i}\right)
\nabla\rho_{i} \label{p4}%
\end{equation}
and%
\begin{equation}
Hess\left(  u_{i}\right)  =R^{2}\overline{\varphi}^{\prime\prime}\left(
R^{2}\rho_{i}\right)  \nabla\rho_{i}\otimes\nabla\rho_{i}+\overline{\varphi
}^{\prime}\left(  R^{2}\rho_{i}\right)  Hess\left(  \rho_{i}\right)  .
\label{p5}%
\end{equation}

From the definition of $\overline{\varphi}$ and the second inequality in
$\left(  \ref{p1}\right)  $, we know $u_{i}=\rho_{i}$. By $\left(
\ref{p1}\right)  $, we have%
\[%
%TCIMACRO{\dint \limits_{B\left(  p_{i},4R\right)  \backslash B\left(
%p_{i},\frac{2}{R}\right)  }}%
%BeginExpansion
{\displaystyle\int\limits_{B\left(  p_{i},4R\right)  \backslash B\left(
p_{i},\frac{2}{R}\right)  }}
%EndExpansion
\left(  \left\vert \nabla u_{i}-\frac{1}{2}\nabla r_{i}^{2}\right\vert
^{2}+\left\vert Hess\left(  u_{i}\right)  -g_{i}\right\vert ^{2}\right)
<\Phi\left(  \frac{1}{R},\frac{1}{i}|n\right)  .
\]

By (\ref{p3}), (\ref{p6}), and (\ref{p4}), we obtain%
\[%
%TCIMACRO{\dint \limits_{B\left(  p_{i},\frac{2}{R}\right)  }}%
%BeginExpansion
{\displaystyle\int\limits_{B\left(  p_{i},\frac{2}{R}\right)  }}
%EndExpansion
\left\vert \nabla u_{i}-\frac{1}{2}\nabla r_{i}^{2}\right\vert ^{2}%
<\Phi\left(  \frac{1}{R},\frac{1}{i}|n\right)  .
\]

From the inequalities (\ref{p5}), (\ref{p3}), (\ref{p6}), and (\ref{p1}), it
could be derived that%
\[%
\begin{array}
[c]{cl}
&
%TCIMACRO{\dint \limits_{B\left(  p_{i},\frac{2}{R}\right)  }}%
%BeginExpansion
{\displaystyle\int\limits_{B\left(  p_{i},\frac{2}{R}\right)  }}
%EndExpansion
\left\vert Hess\left(  u_{i}\right)  -g_{i}\right\vert ^{2}\\
\leq & 2%
%TCIMACRO{\dint \limits_{B\left(  p_{i},\frac{2}{R}\right)  \backslash B\left(
%p_{i},\frac{1}{5R}\right)  }}%
%BeginExpansion
{\displaystyle\int\limits_{B\left(  p_{i},\frac{2}{R}\right)  \backslash
B\left(  p_{i},\frac{1}{5R}\right)  }}
%EndExpansion
\left\vert Hess\left(  u_{i}\right)  \right\vert ^{2}+\Phi\left(  \frac{1}%
{R}|n\right) \\
\leq & C(n)%
%TCIMACRO{\dint \limits_{B\left(  p_{i},\frac{2}{R}\right)  \backslash B\left(
%p_{i},\frac{1}{5R}\right)  }}%
%BeginExpansion
{\displaystyle\int\limits_{B\left(  p_{i},\frac{2}{R}\right)  \backslash
B\left(  p_{i},\frac{1}{5R}\right)  }}
%EndExpansion
\left\vert Hess\left(  \rho_{i}\right)  \right\vert ^{2}+\Phi\left(  \frac
{1}{R},\frac{1}{i}|n\right) \\
\leq & C(n)%
%TCIMACRO{\dint \limits_{B\left(  p_{i},\frac{2}{R}\right)  \backslash B\left(
%p_{i},\frac{1}{5R}\right)  }}%
%BeginExpansion
{\displaystyle\int\limits_{B\left(  p_{i},\frac{2}{R}\right)  \backslash
B\left(  p_{i},\frac{1}{5R}\right)  }}
%EndExpansion
\left\vert Hess\left(  \rho_{i}\right)  -g_{i}\right\vert ^{2}+\Phi\left(
\frac{1}{R},\frac{1}{i}|n\right) \\
< & \Phi\left(  \frac{1}{R},\frac{1}{i}|n\right)  .
\end{array}
\]

Hence the first inequality in (\ref{p7}) is deduced. As for the second
inequality in (\ref{p7}) and the estimate (\ref{p8}), they could be also
observed from (\ref{p1}), (\ref{p6}), (\ref{p4}), and (\ref{p3}). So
\textbf{Claim 5.1.} is accomplished.

Consider a smooth function $\varphi$ from $\left[  0,+\infty\right)  $ to
$\left[  0,+\infty\right)  $ with compact support such that%
\[
\varphi\left(  t\right)  =\left\{
\begin{array}
[c]{cl}%
t & \mathrm{on}\text{ }\left[  0,1\right]  ,\\
0 & \mathrm{on}\text{ }\left[  2,+\infty\right)  ,
\end{array}
\right.
\]
and $\left\vert \varphi\right\vert ,\left\vert \varphi^{\prime}\right\vert
,\left\vert \varphi^{\prime\prime}\right\vert \leq C\left(  n\right)  $. Let
\[
v_{i}\left(  y\right)  =3R^{2}\varphi\left(  \frac{u_{i}\left(  y\right)
}{3R^{2}}\right)
\]
and%
\[
v_{i}\left(  z,t\right)  =v_{i,t}\left(  z\right)  =%
%TCIMACRO{\dint \limits_{M}}%
%BeginExpansion
{\displaystyle\int\limits_{M}}
%EndExpansion
H_{i}\left(  z,y,t\right)  v_{i}\left(  y\right)  ,
\]
where $H_{i}\left(  z,y,t\right)  $ is the heat kernel of $\left(  M_{i}%
,g_{i}\right)  $. It is clear that the support of $v_{i}$ is contained in
$B\left(  p_{i},4R\right)  $. By the similar argument as in Theorem \ref{t2},
it can be showed that

\textbf{Claim 5.2 :}
\[
\left(  v_{i,1}\left(  z\right)  \right)  _{\alpha\overline{\beta}}\geq
c\left(  n,v\right)  \left(  g_{i}\right)  _{\alpha\overline{\beta}}>0
\]
on $B\left(  p_{i},\frac{R}{10}\right)  $ for sufficiently large $R$ and $i$.

Define $q_{i}\left(  x\right)  $ to be%
\[
4n\left[  \log\left(
%TCIMACRO{\dsum \limits_{s\in I_{n}}}%
%BeginExpansion
{\displaystyle\sum\limits_{s\in I_{n}}}
%EndExpansion
\left\vert w_{s}^{i}\left(  y\right)  \right\vert ^{2}\right)  \lambda\left(
4\frac{%
%TCIMACRO{\dsum \limits_{s\in I_{n}}}%
%BeginExpansion
{\displaystyle\sum\limits_{s\in I_{n}}}
%EndExpansion
\left\vert w_{s}^{i}\left(  y\right)  \right\vert ^{2}}{\delta^{2}}\right)
+\log\left(
%TCIMACRO{\dsum \limits_{s\in I_{n}}}%
%BeginExpansion
{\displaystyle\sum\limits_{s\in I_{n}}}
%EndExpansion
\left\vert v_{s}^{i}\left(  y\right)  \right\vert ^{2}\right)  \lambda\left(
4\frac{%
%TCIMACRO{\dsum \limits_{s\in I_{n}}}%
%BeginExpansion
{\displaystyle\sum\limits_{s\in I_{n}}}
%EndExpansion
\left\vert v_{s}^{i}\left(  y\right)  \right\vert ^{2}}{\delta^{2}}\right)
\right]  ,
\]
where $\lambda$ is a cut-off function from $\left[  0,+\infty\right)  $ to
$\left[  0,+\infty\right)  $ with compact support and
\[
\lambda\left(  t\right)  =\left\{
\begin{array}
[c]{cl}%
1 & \mathrm{on}\text{ }\left[  0,1\right]  ,\\
0 & \mathrm{on}\text{ }\left[  2,+\infty\right)  .
\end{array}
\right.
\]

It is not difficult to see that supp$\left(  q_{i}\left(  x\right)  \right)
\subseteq B(y_{i},\delta)\cup B(z_{i},\delta)\subset B\left(  p_{i},2\right)
$. And we can show that

\textbf{Claim 5.3 :}
\[
i\partial_{B}\overline{\partial}_{B}q_{i}\geq-C\left(  n,\delta\right)
d\theta_{i}.
\]
Furthermore, $\exp\left(  -q_{i}\right)  $ is not locally integrable at
$y_{i}$ and $z_{i}$. Note that $\left(  q_{i}\right)  _{0}=0$. \

Now we give the proof of \textbf{Claim 5.3. }with the help of the following
three estimates%
\[
\left\{
\begin{array}
[c]{l}%
\left\vert i\partial_{B}\overline{\partial}_{B}\left\vert w_{s}^{i}\right\vert
^{2}\right\vert \leq\left\vert \partial_{B}w_{s}^{i}\wedge\overline
{\partial_{B}w_{s}^{i}}\right\vert \leq\left\vert \nabla w_{s}^{i}\right\vert
^{2}\leq C\left(  n\right)  \text{ }\mathrm{in}\text{ }B(y_{i},\delta),\\%
%TCIMACRO{\dsum \limits_{s\in I_{n}}}%
%BeginExpansion
{\displaystyle\sum\limits_{s\in I_{n}}}
%EndExpansion
\left\vert w_{s}^{i}\left(  y\right)  \right\vert ^{2}\in\left[  \frac{1}%
{4}\delta^{2},\frac{1}{2}\delta^{2}\right]  \text{ }\mathrm{if}\text{ }%
y\in\left[  \lambda^{\prime}\left(  \frac{4}{\delta^{2}}%
%TCIMACRO{\dsum \limits_{s\in I_{n}}}%
%BeginExpansion
{\displaystyle\sum\limits_{s\in I_{n}}}
%EndExpansion
\left\vert w_{s}^{i}\right\vert ^{2}\right)  \neq0\right]  ,\\
i\partial_{B}\overline{\partial}_{B}\log\left(
%TCIMACRO{\dsum \limits_{s\in I_{n}}}%
%BeginExpansion
{\displaystyle\sum\limits_{s\in I_{n}}}
%EndExpansion
\left\vert w_{s}^{i}\right\vert ^{2}\right)  \geq0\text{ }\mathrm{in}\text{
}\mathrm{the}\text{ }\mathrm{current}\text{ }\mathrm{sense.}%
\end{array}
\right.
\]
The first part is completed by the routine calculation. As for the second one,
by the facts that, in $B(y_{i},\frac{\delta}{10})$,
\[
\exp\left(  -q_{i}\right)  =\frac{1}{\left(
%TCIMACRO{\dsum \limits_{s\in I_{n}}}%
%BeginExpansion
{\displaystyle\sum\limits_{s\in I_{n}}}
%EndExpansion
\left\vert w_{s}^{i}\left(  y\right)  \right\vert ^{2}\right)  ^{4n}}%
\]
and%
\[
w_{s}^{i}\left(  y_{i}\right)  =0,
\]
for any $s\in I_{n}$, it enables us to see that $\exp\left(  -q_{i}\right)  $
is not locally integrable at $y_{i}$. By the similar argument, we also have
that $\exp\left(  -q_{i}\right)  $ is not locally integrable at $z_{i}$. This
implies \textbf{Claim 5.3.}

We can observe that, in $B\left(  p_{i},\frac{R}{15}\right)  $,
\[
i\partial_{B}\overline{\partial}_{B}\left(  q_{i}+C\left(  n,v,\delta\right)
v_{i,1}\right)  \geq d\theta_{i}%
\]
for some positive constant $C\left(  n,v,\delta\right)  $. Set
\[
\psi_{i}=q_{i}+C\left(  n,v,\delta\right)  v_{i,1}.
\]
As in Theorem \ref{t2}, there exists a sufficiently small number $\epsilon
_{0}=\epsilon_{0}\left(  n,v\right)  >0$ such that
\begin{equation}
4\sup_{B\left(  p_{i},\epsilon_{0}\frac{R}{20}\right)  }v_{i,1}<\inf_{\partial
B\left(  p_{i},\frac{R}{20}\right)  }v_{i,1}, \label{p9}%
\end{equation}
for sufficiently large $R$. From now on, we freeze the value of $R=R\left(
n,v\right)  $ such that $R$ satisfies \textbf{Claim 5.2.,} (\ref{p9}), and
$\epsilon_{0}\frac{R}{20}>4$. Let $\Omega_{i}$ be the connected component of
\[
\left\{  y\in B\left(  p_{i},\frac{R}{20}\right)  |v_{i,1}\left(  y\right)
<2\sup_{B\left(  p_{i},\epsilon_{0}\frac{R}{20}\right)  }v_{i,1}\right\}
\]
which contains $B\left(  p_{i},\epsilon_{0}\frac{R}{20}\right)  $. It is clear
that $\Omega_{i}$ is relatively compact in $B\left(  p_{i},\frac{R}%
{20}\right)  $ and $\left(  v_{i,1}\left(  z\right)  \right)  _{\alpha
\overline{\beta}}\geq c\left(  n,v\right)  \left(  g_{i}\right)
_{\alpha\overline{\beta}}>0$ on $\Omega_{i}$. Consider a smooth function
$f_{i}$ with compact support, $Tf_{i}=0$,
\[
f_{i}=\left\{
\begin{array}
[c]{cl}%
1 & \mathrm{on}\text{ }B(y_{i},\frac{\delta}{4}),\\
0 & \mathrm{on}\text{ }B(y_{i},\delta)^{c},
\end{array}
\right.
\]
and
\[
\left\vert \nabla f_{i}\right\vert \leq C\left(  n,v,\delta\right)  .
\]
By CR H\"{o}rmander $L^{2}$-estimate, we solve the equation%
\[
\overline{\partial}_{B}h_{i}=\overline{\partial}_{B}f_{i}%
\]
in $\Omega_{i}$ with
\begin{equation}
\left\Vert h_{i}\right\Vert _{L^{2}\left(  \Omega_{i},e^{-\psi_{i}}d\mu
_{i}\right)  }\leq\left\Vert \overline{\partial}_{B}f_{i}\right\Vert
_{L^{2}\left(  \Omega_{i},e^{-\psi_{i}}d\mu_{i}\right)  }\leq C\left(
n,v,\delta\right)  . \label{p10}%
\end{equation}
Because $\exp\left(  -q_{i}\right)  $ is not locally integrable at $y_{i}$ and
$z_{i}$, we see that $h_{i}\left(  y_{i}\right)  =0=h_{i}\left(  z_{i}\right)
$. Let $\mu_{i}=f_{i}-h_{i}$. Hence we have $\mu_{i}\left(  y_{i}\right)  =1$
and $\mu_{i}\left(  z_{i}\right)  =0$. It means that $\mu_{i}$ is not constant
on $\Omega_{i}$. By definition, we could see that
\begin{equation}
\psi_{i}\leq C\left(  n,v,\delta\right)  \label{p11}%
\end{equation}
in $B\left(  p_{i},3\right)  $. From the definition of $f_{i}$, (\ref{p11}),
(\ref{p10}), the mean value inequality gives us that
\[
\left\vert \mu_{i}\right\vert \leq C\left(  n,v,\delta\right)
\]
in $B\left(  p_{i},2\right)  $. More precisely, we have
\[%
\begin{array}
[c]{cl}
&
%TCIMACRO{\dint \limits_{B\left(  p_{i},3\right)  }}%
%BeginExpansion
{\displaystyle\int\limits_{B\left(  p_{i},3\right)  }}
%EndExpansion
\left\vert \mu_{i}\right\vert ^{2}\\
\leq & 2%
%TCIMACRO{\dint \limits_{B\left(  p_{i},3\right)  }}%
%BeginExpansion
{\displaystyle\int\limits_{B\left(  p_{i},3\right)  }}
%EndExpansion
\left(  \left\vert f_{i}\right\vert ^{2}+\left\vert h_{i}\right\vert
^{2}\right) \\
\leq & C\left(  n,v,\delta\right)  +%
%TCIMACRO{\dint \limits_{B\left(  p_{i},3\right)  }}%
%BeginExpansion
{\displaystyle\int\limits_{B\left(  p_{i},3\right)  }}
%EndExpansion
\left\vert h_{i}\right\vert ^{2}\\
\leq & C\left(  n,v,\delta\right)  +C\left(  n,v,\delta\right)  \left\Vert
h_{i}\right\Vert _{L^{2}\left(  \Omega_{i},e^{-\psi_{i}}d\mu_{i}\right)  }\\
\leq & C\left(  n,v,\delta\right)  .
\end{array}
\]
Therefore, the CR-holomorphic function $v_{i}^{\ast}=$ $\mu_{i}-\mu_{i}\left(
p_{i}\right)  $ on $B\left(  p_{i},2\right)  $ is uniformly bounded. Let
\[
M_{i}^{\prime}\left(  r\right)  =\sup_{B\left(  p_{i},r\right)  }\left\vert
v_{i}^{\ast}\right\vert .
\]
Then we have $M_{i}^{\prime}\left(  2\right)  \leq C\left(  n,v,\delta\right)
$. By $\mu_{i}\left(  y_{i}\right)  =1$ and $\mu_{i}\left(  z_{i}\right)  =0$,
we obtain $M_{i}^{\prime}\left(  1\right)  \geq\frac{1}{2}$. So
\begin{equation}
\frac{M_{i}^{\prime}\left(  2\right)  }{M_{i}^{\prime}\left(  1\right)  }\leq
C\left(  n,v,\delta\right)  . \label{p12}%
\end{equation}

We consider the rescale function $\mu_{i}^{\ast}=\alpha_{i}v_{i}^{\ast}$ in
$B\left(  p,2s_{i}\right)  $ so that
\[
\left\Vert \mu_{i}^{\ast}\right\Vert _{L^{2}\left(  B\left(  p,2\right)
,d\mu_{i}\right)  }=1.
\]
This implies that, in $B\left(  p,1\right)  $,
\begin{equation}
\left\vert \mu_{i}^{\ast}\right\vert \leq C\left(  n,v\right)  . \label{p13}%
\end{equation}

If we define
\[
M_{i}\left(  r\right)  =\sup_{B\left(  p,r\right)  }\left\vert \mu_{i}^{\ast
}\right\vert ,
\]
then, by the CR three-circle theorem (\ref{314}), we have $\frac{M_{i}\left(
2r\right)  }{M_{i}\left(  r\right)  }$ is monotonic increasing for $0<r\leq
s_{i}.$ This and (\ref{p12}) will imply
\begin{equation}
\frac{M_{i}\left(  2r\right)  }{M_{i}\left(  r\right)  }\leq C\left(
n,v,\delta\right)  \label{p14}%
\end{equation}
for $r\in\left(  0,s_{i}\right]  $. Furthermore, it follows from (\ref{p14})
that%
\[
\left\{
\begin{array}
[c]{l}%
\ln M_{i}\left(  2^{k}\right)  -\ln M_{i}\left(  2^{k-1}\right)  \leq C\left(
n,v,\delta\right)  ,\\
\multicolumn{1}{c}{\vdots}\\
\ln M_{i}\left(  2\right)  -\ln M_{i}\left(  1\right)  \leq C\left(
n,v,\delta\right)  ,
\end{array}
\right.
\]
where $k$ is chosen as $\left[  \log_{2}r\right]  +1$. Then, with the help of
(\ref{p13}), we obtain%
\[
\ln M_{i}\left(  r\right)  \leq\ln M_{i}\left(  2^{k}\right)  \leq C\left(
n,v,\delta\right)  \left(  k+1\right)  .
\]
From the last inequality, it is easy to deduce, for all $i$
\begin{equation}
M_{i}\left(  r\right)  \leq C\left(  n,v,\delta\right)  \left(  r^{C\left(
n,v,\delta\right)  }+1\right)  , \label{p15}%
\end{equation}
for $s_{i}\geq r$. Passing to a subsequence $i\rightarrow\infty$, $\mu
_{i}^{\ast}$ converges uniformly on each compact set to a nonconstant
CR-holomorphic function $\mu$ of polynomial growth with $\mu(p)=0$ and
\[%
%TCIMACRO{\dint \limits_{B\left(  p,2\right)  }}%
%BeginExpansion
{\displaystyle\int\limits_{B\left(  p,2\right)  }}
%EndExpansion
\left\vert \mu\right\vert ^{2}=1.
\]
This completes the proof of Theorem \ref{T1}.
\end{proof}

\appendix

\section{ CR analogue of H\"{o}rmander L$^{2}$-estimate {}}

Let $\phi$ be an $\left(  p,q\right)  $-form and denoted by
\[
\phi=\phi_{\alpha_{1}\alpha_{2}...\alpha_{p}\bar{\beta}_{1}\bar{\beta}%
_{2}...\bar{\beta}_{q}}\theta^{\alpha_{1}}\wedge...\wedge\theta^{\alpha_{p}%
}\wedge\theta^{\bar{\beta}_{1}}\wedge\theta^{\bar{\beta}_{2}}\wedge
...\wedge\theta^{\bar{\beta}_{q}}.
\]
For abbreviation, we denote as $\phi=\phi_{A\bar{B}}\theta^{A}\wedge
\theta^{\bar{B}}$, where $A$ and $\overline{B}$ are multiple index $A=\left(
\alpha_{1},\alpha_{2},...,\alpha_{p}\right)  $ and $\overline{B}=\left(
\bar{\beta}_{1},\bar{\beta}_{2},...,\bar{\beta}_{q}\right)  $ respectively. By
taking exterior derivative, we obtain%
\[
d\phi=d_{b}\phi+\mathbf{T}\phi:=\partial_{b}\phi+\bar{\partial}_{b}%
\phi+\mathbf{T}\phi,
\]
where $\partial_{b}\phi$ is the $\left(  p+1,q\right)  $-form part of
$d_{b}\phi$ \ and $\bar{\partial}_{b}\phi$ is the $\left(  p,q+1\right)
$-form part of $d_{b}\phi$ , $\mathbf{T}\phi$ is the form spanned by basis
$\theta\wedge\theta^{A}\wedge\theta^{\bar{B}}$. Note that
\[
\overline{\partial}_{b}^{2}=0.
\]

Before we go further, let us first recall the notion as in \cite{fow} for a
Sasakain $(2n+1)$-manifold. Let $\{U_{\alpha}\}_{\alpha\in A}$ be an open
covering of the Sasakian manifold $M^{2n+1}$with the adapted metric
$g^{\lambda}$ and $\pi_{\alpha}:$ $U_{\alpha}\rightarrow V_{\alpha}\subset%
%TCIMACRO{\U{2102} }%
%BeginExpansion
\mathbb{C}
%EndExpansion
^{n}$ submersion such that $\pi_{\alpha}\circ\pi_{\beta}^{-1}:\pi_{\beta
}(U_{\alpha}\cap U_{\beta})\rightarrow\pi_{\alpha}(U_{\alpha}\cap U_{\beta})$
is biholomorphic. On each $V_{\alpha},$ there is a canonical isomorphism
\[
d\pi_{\alpha}:D_{p}\rightarrow T_{\pi_{\alpha}(p)}V_{\alpha}%
\]
for any $p\in U_{\alpha},$ where $D=\ker\theta\subset TM.$ Since $\mathbf{T}$
generates isometries, the restriction of the Sasakian metric $g$ to $D$ gives
a well-defined Hermitian metric $g_{\alpha}^{T}$ on $V_{\alpha}.$ This
Hermitian metric in fact is K\"{a}hler. More precisely, let $(z^{1}%
,z^{2},\cdot\cdot\cdot,z^{n})$ be the local holomorphic coordinates on
$V_{\alpha}$. We pull back these to $U_{\alpha}$ and still write the same. Let
$\widetilde{t}$ be the coordinate along the leaves with $\mathbf{T}%
=\frac{\partial}{\partial\widetilde{t}}.$ Then we have the local coordinate
$x=(\widetilde{x},\widetilde{t})=(z^{1},z^{2},\cdot\cdot\cdot,z^{n}%
,\widetilde{t})$ on $U_{\alpha}\ $and $(D\otimes%
%TCIMACRO{\U{2102} }%
%BeginExpansion
\mathbb{C}
%EndExpansion
)^{1,0}$ is spanned by the form
\begin{equation}
Z_{\alpha}=(\frac{\partial}{\partial z^{\alpha}}-\theta(\frac{\partial
}{\partial z^{\alpha}})T),\ \ \ \alpha=1,2,\cdot\cdot\cdot,n. \label{2018}%
\end{equation}
Since $i(T)d\theta=0,$
\begin{equation}
d\theta(Z_{\alpha},\overline{Z_{\beta}})=d\theta(\frac{\partial}{\partial
z^{\alpha}},\frac{\partial}{\partial\overline{z}^{\beta}}). \label{2018b}%
\end{equation}
Then the K\"{a}hler $2$-form $\omega_{\alpha}^{T}$ of the Hermitian metric
$g_{\alpha}^{T}$ on $V_{\alpha},$ which is the same as the restriction of the
Levi form $\frac{1}{2}d\theta$ to $\widetilde{D_{\alpha}}$, the slice
$\{\widetilde{t}=$ \textrm{constant}$\}$ in $U_{\alpha},$ is closed. The
collection of K\"{a}hler metrics $\{g_{\alpha}^{T}\}$ on $\{V_{\alpha}\}$ is
so-called a transverse K\"{a}hler metric. We often refer to $\frac{1}%
{2}d\theta$ on $\widetilde{D}$ as the K\"{a}hler form of the transverse
K\"{a}hler metric $g^{T}$ in the leaf space $V.$ Furthermore, Being an
Sasakian manifold, one observes that the relation between transverse
K\"{a}hler Ricci curvature and pseudohermitian Ricci curvature
\begin{equation}
Ric_{g_{\alpha}^{T}}(X,Y)=Ric(\widetilde{X},\widetilde{Y})+2g_{\alpha}%
^{T}(X,Y), \label{2018c}%
\end{equation}
where $X,$ $Y$ are vector fields on the local leaf spaces $V_{\alpha}$ and
$\widetilde{X},$ $\widetilde{Y}$ are lifts to $D$.

As an example, $\left\{  Z_{\alpha}=\frac{\partial}{\partial z^{\alpha}%
}+i\overline{z}^{\alpha}\frac{\partial}{\partial t}\right\}  _{\alpha=1}^{n}$
is exactly a local frame on in the $(2n+1)$-dimensional Heisenberg group
$\mathbf{H}_{n}$ $=$ $%
%TCIMACRO{\U{2102} }%
%BeginExpansion
\mathbb{C}
%EndExpansion
^{n}\times%
%TCIMACRO{\U{211d} }%
%BeginExpansion
\mathbb{R}
%EndExpansion
$ with the local coordinate $(z,t)$. Here
\[
\theta=dt+i%
%TCIMACRO{\dsum \limits_{\alpha\in I_{n}}}%
%BeginExpansion
{\displaystyle\sum\limits_{\alpha\in I_{n}}}
%EndExpansion
\left(  z^{\alpha}d\overline{z}^{\alpha}-\overline{z}^{\alpha}dz^{\alpha
}\right)
\]
is a pseudohermitian contact structure on $\mathbf{H}_{n}$ and $\mathbf{T}%
=\frac{\partial}{\partial\widetilde{t}}$.

\begin{definition}
Let $\left(  M,J,\theta\right)  $ be a complete noncompact Sasakian
$(2n+1)$-manifold. Define

(i) A $p$-form $\gamma$ is called basic if%
\[
i(\mathbf{T})\gamma=0
\]
and
\[
L_{\mathbf{T}}\gamma=0
\]

(ii) Let $\Lambda_{B}^{p}$ be the sheaf of germs of basic $p$-forms and
\ $\Omega_{B}^{p}$ be the set of all global sections of $\Lambda_{B}^{p}$.
\end{definition}

It is easy to check that $d\gamma$ is basic if $\gamma$ is basic. Set
$d_{B}=d|_{\Omega_{B}^{p}}.$ Then%
\[
d_{B}:\Omega_{B}^{p}\rightarrow\Omega_{B}^{p+1}%
\]
and the corresponding complex $(\Omega_{B}^{\ast},d_{B})$ is called the basic
de Rham complex. The basic form of type $(p,q)$-form with respect to
$(z^{1},z^{2},\cdot\cdot\cdot,z^{n},\widetilde{t}_{1})$ in $U_{\alpha}$
\[
\gamma=\gamma_{i_{1}...i_{p}\overline{j}_{1}...\overline{j}_{q}}dz^{i_{1}%
}\wedge...\wedge dz^{i_{p}}\wedge d\overline{z}^{j_{1}}\wedge...\wedge
d\overline{z}^{j_{q}}%
\]
is also of type $(p,q)$-form with respect to $(w^{1},w^{2},\cdot\cdot
\cdot,w^{n},\widetilde{t}_{2})$ in $U_{\beta}$ and does not depend on
$\widetilde{t}.$ We then have the well-defined operators%
\[
\partial_{B}:\Lambda_{B}^{p,q}\rightarrow\Lambda_{B}^{p+1,q}%
\]
and
\[
\overline{\partial}_{B}:\Lambda_{B}^{p,q}\rightarrow\Lambda_{B}^{p,q+1}.
\]
Now for
\[
d_{B}=\overline{\partial}_{B}+\partial_{B}\text{ \textrm{and }\ }d_{B}%
^{c}=\frac{i}{2}(\overline{\partial}_{B}-\partial_{B}),
\]
we have
\[
d_{B}d_{B}^{c}=i\partial_{B}\overline{\partial}_{B}\text{ \ \ \ \textrm{and}
\ \ }d_{B}^{2}=0=(d_{B}^{c})^{2}.
\]
We also define the basic Laplacian
\[
\Delta_{B}=d_{B}^{\ast}d_{B}+d_{B}d_{B}^{\ast}.
\]
Here $d_{B}^{\ast}$ $:\Omega_{B}^{p+1}\rightarrow\Omega_{B}^{p}$ is the
adjoint operator of $d_{B}.$ As a result of El Kacimi-Alaoui \cite{eka}, one
has the expected isomorphisms between basic cohomology groups and the space of
basic harmonic forms. Moreover, we can work out the basic Hodge theory in the
local leaf space $V$ together with the transverse K\"{a}hler structure and the
transverse K\"{a}hler metric $g^{T}$.

More precisely, we consider a CR-holomorphic vector bundle $(E,\overline
{\partial}_{b})$ over a strictly pseudoconvex CR $(2n+1)$-manifold
$(M,T^{1,0}(M),\theta)$ and a unique Tanaka connection $D$ on $E$ due to N.
Tanaka (\cite{ta}). Define the global $L^{2}$-norm for any $L^{2}$ section $s$
of $E$
\[
\int_{M}||s(x)||^{2}d\mu
\]
where $||s(x)||^{2}=\left\langle s(x),s(x)\right\rangle _{L_{\theta}}$ is the
pointwise Hermitian norm and $d\mu=\theta\wedge(d\theta)^{n}$. The sub-Laplace
Beltrami operator associated to this connection is defined by
\[
\Delta=D^{\ast}D+DD^{\ast}%
\]
where $D^{\ast}$ is the adjoint of $D$ with respect to the $L^{2}$-norm as
above. The Tanaka connection
\[
D=D^{1,0}+D^{0,1}%
\]
have decomposition with
\[
D^{0,1}=\overline{\partial}_{b}.
\]
We define the complex sub-Laplace operators
\[
\Delta^{\prime}=(D^{\prime})^{\ast}D^{\prime}+D^{\prime}(D^{\prime})^{\ast}%
\]
with $D^{\prime}=D^{1,0}$ and
\[
\Delta_{\overline{\partial}_{b}}=\overline{\partial}_{b}^{\ast}\overline
{\partial}_{b}+\overline{\partial}_{b}\overline{\partial}_{b}^{\ast}.
\]
Now we will work on an noncompact Sasakian manifold $M$. In fact we consider
all operators on the sheaf of germs of basic $p$-forms $\Lambda_{B}^{p}$ and
\ the set of all global sections $\Omega_{B}^{p}$ of $\Lambda_{B}^{p}$ with
the complex basic sub-Laplacian
\[
\Delta_{\overline{\partial}_{B}}=\overline{\partial}_{B}^{\ast}\overline
{\partial}_{B}+\overline{\partial}_{B}\overline{\partial}_{B}^{\ast}.
\]

\begin{lemma}
(\textbf{CR Bochner-Kodaira-Nakano identity}) Let $(M,T^{1,0}(M),\theta)$ be a
strictly pseudoconvex Sasakian $(2n+1)$-manifold and $D$ be the Tanaka
connection in a CR-holomorphic vector bundle $E$ over $M$. Then the complex
sub-Laplace operators $\Delta^{\prime}$ and $\Delta_{\overline{\partial}_{B}}$
acting on $E$-valued basic forms satisfy the identity%
\begin{equation}
\Delta_{\overline{\partial}_{B}}=\Delta^{\prime}+[i\Theta(E),\Lambda],
\label{2020AA}%
\end{equation}
where $\Delta_{\overline{\partial}_{B}}=\overline{\partial}_{B}^{\ast
}\overline{\partial}_{B}+\overline{\partial}_{B}\overline{\partial}_{B}^{\ast
}.$ Here $\Theta(E):=D^{2}s$ is the curvature operator.
\end{lemma}

\begin{proof}
Let $L$ be the Lefschetz operator defined by $Ls=\frac{1}{2}d\theta\wedge s$
and $\Lambda=L^{\ast}$ its adjoint operator. It follows from \cite{ti} (see
also \cite{de}) that the following transverse K\"{a}hler identities hold :
\[
\lbrack\overline{\partial}_{B}^{\ast},L]=iD^{\prime};\ \text{\ \ }[D^{\prime
}{}^{\ast},L]=-i\overline{\partial}_{B}%
\]
and
\[
\lbrack\Lambda,\overline{\partial}_{B}]=-iD^{\prime\ast};\text{\ \ }%
[\Lambda,D^{\prime}{}]=i\overline{\partial}_{B}^{\ast}.
\]
Hence we have
\[
\Delta_{\overline{\partial}_{B}}=[\overline{\partial}_{B},\overline{\partial
}_{B}^{\ast}]=-i[\overline{\partial}_{B},[\Lambda,D^{\prime}]].
\]
On the other hand, by the Jacobi identity we have
\[
\lbrack\overline{\partial}_{B},[\Lambda,D^{\prime}]]=[\Lambda,[D^{\prime
},\overline{\partial}_{B}]]+[D^{\prime},[\overline{\partial}_{B}%
,\Lambda]]=[\Lambda,\Theta(E)]+i[D^{\prime},D^{\prime\ast}],
\]
where we use $[D^{\prime},\overline{\partial}_{B}]=D^{2}=\Theta(E)$. Then we
derive the CR analogue of Bochner-Kodaira-Nakano identity.
\end{proof}

Then, based on this CR Bochner-Kodaira-Nakano identity (\ref{2020AA}), it is
straightforward from \cite{h} and \cite{de} that one can derive the CR
analogue of H\"{o}rmander $L^{2}$-estimate Proposition \ref{P21}.

\section{ Maximum Principle for the CR Heat Equation}

In order to prove Theorem \ref{t1}, we derive several results by applying the
similar method as in our previous papers (\cite{cchl}, \cite{chl1}).

\begin{lemma}
\label{l1} Let $\left(  M,J,\theta\right)  $ be a complete noncompact Sasakian
$\left(  2n+1\right)  $-manifold of nonnegative pseudohermitian Ricci
curvature with
\[
0\leq u\left(  x\right)  \leq\exp\left(  ar\left(  x\right)  +b\right)
\]
for some $a,b>0,$ $u\in C_{c}^{\infty}\left(  M\right)  .$ We denote
\[
v\left(  x,t\right)  =%
%TCIMACRO{\dint \limits_{M}}%
%BeginExpansion
{\displaystyle\int\limits_{M}}
%EndExpansion
H\left(  x,y,t\right)  u\left(  y\right)  d\mu\left(  y\right)
\]
on $M\times\left[  0,\infty\right)  $ with the CR heat kernel $H\left(
x,y,t\right)  $. Then for any positive numbers $\epsilon>0$ and $T>0$, there
exist $C_{1}\left(  n,\epsilon\right)  >0$ and $C_{2}\left(  n,a,b,\epsilon
,T\right)  >0$ such that
\[
C_{1}\left(  n,\epsilon\right)  \inf_{B\left(  x,\epsilon r\left(  x\right)
\right)  }u\leq v\left(  x,t\right)  \leq C_{2}\left(  n,a,b,\epsilon
,T\right)  +\sup_{B\left(  x,\epsilon r\left(  x\right)  \right)  }u.
\]
Here $\left(  x,t\right)  \in r^{-1}\left(  \left[  \sqrt{T},+\infty\right)
\right)  \times\left[  0,T\right]  $ and $r\left(  x\right)  =d_{cc}\left(
x,o\right)  $ is the Carnot-Carath\'{e}odory distance $d_{cc}$ between $x$ and
$o$ for some fixed point $o\in M$ and $B\left(  x,r\right)  $ is the ball with
respect to $d_{cc}$.
\end{lemma}

\begin{proof}
If $s\geq\epsilon r\left(  x\right)  $, then
\begin{equation}
\widehat{\left\vert u\right\vert }_{B\left(  x,s\right)  }\leq\frac
{V_{o}\left(  \left(  1+\frac{1}{\epsilon}\right)  s\right)  }{V_{x}\left(
s\right)  }\widehat{\left\vert u\right\vert }_{B\left(  o,\left(  1+\frac
{1}{\epsilon}\right)  s\right)  }\leq C\left(  n,\epsilon\right)
\widehat{\left\vert u\right\vert }_{B\left(  o,\left(  1+\frac{1}{\epsilon
}\right)  s\right)  }. \label{m1}%
\end{equation}
Where we use the fact that
\[%
\begin{array}
[c]{ccl}%
V_{o}\left(  \left(  1+\frac{1}{\epsilon}\right)  s\right)  & \leq &
V_{o}\left(  \left(  1+\frac{1}{\epsilon}\right)  s+r\left(  x\right)  \right)
\\
& \leq & V_{o}\left(  \left(  1+\frac{2}{\epsilon}\right)  s\right) \\
& \leq & C\left(  n\right)  \left(  1+\frac{2}{\epsilon}\right)
^{2C_{9}\left(  n\right)  }V_{x}\left(  s\right) \\
& = & C\left(  n,\epsilon\right)  V_{x}\left(  s\right)  .
\end{array}
\]

Here $V_{x}\left(  s\right)  $ and $\widehat{\left\vert u\right\vert }%
_{B_{x}\left(  s\right)  }$ denote the volume of the Carnot-Carath\'{e}odory
ball $B\left(  x,s\right)  =B_{x}\left(  s\right)  $ with respect to the
volume element $d\mu=\theta\wedge\left(  d\theta\right)  ^{n}$ and the average
of the absolute value of the function $u$ over $B_{x}\left(  s\right)  $
respectively. Due to
\[
r\left(  x\right)  \geq\sqrt{T},
\]
we have%
\begin{equation}
\frac{V_{x}\left(  s\right)  }{V_{x}\left(  \sqrt{t}\right)  }\leq\frac
{V_{x}\left(  \frac{s}{\epsilon}\right)  }{V_{x}\left(  \sqrt{t}\right)  }\leq
C\left(  n,\epsilon\right)  \left(  \frac{s}{\sqrt{t}}\right)  ^{2C_{9}\left(
n\right)  }. \label{m2}%
\end{equation}

It follows from (\ref{m1}), (\ref{m2}), and
\[
\exp\left(  a\left(  1+\frac{1}{\epsilon}\right)  s+b-\frac{C_{5}\left(
n\right)  }{2t}s^{2}\right)  \leq C\left(  n,a,b,\epsilon,T\right)
\]
for $s\in\left[  0,+\infty\right)  $, that we are able to derive%
\[%
\begin{array}
[c]{cl}
& \left\vert v\left(  x,t\right)  -%
%TCIMACRO{\dint \limits_{B\left(  x,\epsilon r\left(  x\right)  \right)  }}%
%BeginExpansion
{\displaystyle\int\limits_{B\left(  x,\epsilon r\left(  x\right)  \right)  }}
%EndExpansion
H\left(  x,y,t\right)  u\left(  y\right)  d\mu\left(  y\right)  \right\vert \\
\leq &
%TCIMACRO{\dint \limits_{M\backslash B\left(  x,\epsilon r\left(  x\right)
%\right)  }}%
%BeginExpansion
{\displaystyle\int\limits_{M\backslash B\left(  x,\epsilon r\left(  x\right)
\right)  }}
%EndExpansion
H\left(  x,y,t\right)  u\left(  y\right)  d\mu\left(  y\right) \\
\leq & \frac{C\left(  n\right)  }{V_{x}\left(  \sqrt{t}\right)  }%
\int_{\epsilon r\left(  x\right)  }^{+\infty}\exp\left(  -C_{5}\left(
n\right)  \frac{s^{2}}{t}\right)  \left(
%TCIMACRO{\dint \limits_{\partial B_{x}\left(  s\right)  }}%
%BeginExpansion
{\displaystyle\int\limits_{\partial B_{x}\left(  s\right)  }}
%EndExpansion
\left\vert u\right\vert \right)  ds\\
\leq & C\left(  n\right)  \int_{\epsilon r\left(  x\right)  }^{+\infty}%
\frac{V_{x}\left(  s\right)  }{V_{x}\left(  \sqrt{t}\right)  }\left(
\left\vert u\right\vert _{B\left(  x,s\right)  }\right)  \exp\left(
-C_{5}\left(  n\right)  \frac{s^{2}}{t}\right)  d\left(  \frac{s^{2}}%
{t}\right) \\
\leq & C\left(  n,\epsilon\right)  \int_{\epsilon r\left(  x\right)
}^{+\infty}\left(  \frac{s}{\sqrt{t}}\right)  ^{2C_{9}\left(  n\right)
}\left(  \left\vert u\right\vert _{B\left(  o,\left(  1+\frac{1}{\epsilon
}\right)  s\right)  }\right)  \exp\left(  -C_{5}\left(  n\right)  \frac{s^{2}%
}{t}\right)  d\left(  \frac{s^{2}}{t}\right) \\
\leq & C\left(  n,a,b,\epsilon,T\right)
\end{array}
\]

by the CR heat kernel estimate and the CR volume doubling property
(\cite{ccht}). This implies that%
\[%
\begin{array}
[c]{ccl}%
v\left(  x,t\right)  & \leq & C\left(  n,a,b,\epsilon,T\right)  +%
%TCIMACRO{\dint \limits_{B\left(  x,\epsilon r\left(  x\right)  \right)  }}%
%BeginExpansion
{\displaystyle\int\limits_{B\left(  x,\epsilon r\left(  x\right)  \right)  }}
%EndExpansion
H\left(  x,y,t\right)  u\left(  y\right)  d\mu\left(  y\right) \\
& \leq & C_{2}\left(  n,a,b,\epsilon,T\right)  +\sup_{B\left(  x,\epsilon
r\left(  x\right)  \right)  }u.
\end{array}
\]

On the other hand, it is not difficult to deduce that%
\[%
\begin{array}
[c]{ccl}%
v\left(  x,t\right)  & \geq &
%TCIMACRO{\dint \limits_{B\left(  x,\epsilon r\left(  x\right)  \right)  }}%
%BeginExpansion
{\displaystyle\int\limits_{B\left(  x,\epsilon r\left(  x\right)  \right)  }}
%EndExpansion
H\left(  x,y,t\right)  u\left(  y\right)  d\mu\left(  y\right) \\
& \geq & \frac{C\left(  n\right)  }{V_{x}\left(  \sqrt{t}\right)  }%
%TCIMACRO{\dint \limits_{B\left(  x,\epsilon\sqrt{t}\right)  }}%
%BeginExpansion
{\displaystyle\int\limits_{B\left(  x,\epsilon\sqrt{t}\right)  }}
%EndExpansion
\exp\left(  -C_{7}\left(  n\right)  \frac{d_{cc}^{2}\left(  x,y\right)  }%
{t}\right)  u\left(  y\right)  d\mu\left(  y\right) \\
& \geq & C\left(  n,\epsilon\right)  \frac{V_{x}\left(  \epsilon\sqrt
{t}\right)  }{V_{x}\left(  \sqrt{t}\right)  }\underset{B\left(  x,\epsilon
r\left(  x\right)  \right)  }{\inf}u\\
& \geq & C_{1}\left(  n,\epsilon\right)  \underset{B\left(  x,\epsilon
r\left(  x\right)  \right)  }{\inf}u.
\end{array}
\]

The proof is accomplished.
\end{proof}

\begin{lemma}
\label{l2} Let $\left(  M,J,\theta\right)  $ be a complete noncompact Sasakian
$\left(  2n+1\right)  $-manifold of nonnegative pseudohermitian bisectional
curvature and $v\left(  x,t\right)  $ be a nonnegative solution to the CR heat
equation on $M\times\left[  0,T\right]  $. Then $\left\Vert \eta
_{\alpha\overline{\beta}}\left(  x,t\right)  \right\Vert $ is a subsolution to
the CR heat equation. Furthermore, if%
\begin{equation}%
%TCIMACRO{\dint \limits_{M}}%
%BeginExpansion
{\displaystyle\int\limits_{M}}
%EndExpansion
\exp\left(  -ar^{2}\left(  x\right)  \right)  \left\Vert \eta_{\alpha
\overline{\beta}}\left(  x,0\right)  \right\Vert d\mu\left(  x\right)
<+\infty\label{m3}%
\end{equation}
and%
\begin{equation}
\underset{r\longrightarrow+\infty}{\lim\inf}\int_{0}^{T}%
%TCIMACRO{\dint \limits_{B_{o}\left(  r\right)  }}%
%BeginExpansion
{\displaystyle\int\limits_{B_{o}\left(  r\right)  }}
%EndExpansion
\exp\left(  -ar^{2}\left(  x\right)  \right)  \left\Vert \eta_{\alpha
\overline{\beta}}\left(  x,t\right)  \right\Vert ^{2}dxdt<+\infty\label{m4}%
\end{equation}
for any positive number $a>0$. \ Then%
\[
\left\Vert \eta_{\alpha\overline{\beta}}\left(  x,t\right)  \right\Vert \leq
h\left(  x,t\right)
\]
on $M\times\left[  0,T\right]  $. \ Here $\eta_{\alpha\overline{\beta}%
}=v_{\alpha\overline{\beta}}+v_{\overline{\beta}\alpha}$ and
\begin{equation}
h\left(  x,t\right)  =%
%TCIMACRO{\dint \limits_{M}}%
%BeginExpansion
{\displaystyle\int\limits_{M}}
%EndExpansion
H\left(  x,y,t\right)  \left\Vert \eta_{\alpha\overline{\beta}}\left(
y,0\right)  \right\Vert d\mu\left(  y\right)  \label{m5}%
\end{equation}
on $M\times\left[  0,+\infty\right)  $.
\end{lemma}

\begin{remark}
It is not difficult to observe the existence of the function $h\left(
x,t\right)  $ in $M\times\left[  0,+\infty\right)  $ is ensured by (\ref{m3}).
\end{remark}

\begin{proof}
By \cite[(3.5)]{cftw} and the vanishing pseudohermitian torsion, we see that%
\[
\left(  \frac{\partial}{\partial t}-\Delta_{b}\right)  v_{\alpha
\overline{\beta}}=2R_{\delta\overline{\gamma}\alpha\overline{\beta}}%
v_{\gamma\overline{\delta}}-R_{\alpha\overline{\mu}}v_{\mu\overline{\beta}%
}-R_{\mu\overline{\beta}}v_{\alpha\overline{\mu}}.
\]
Therefore%
\[
\left(  \frac{\partial}{\partial t}-\Delta_{b}\right)  \eta_{\alpha
\overline{\beta}}=2R_{\delta\overline{\gamma}\alpha\overline{\beta}}%
\eta_{\gamma\overline{\delta}}-R_{\alpha\overline{\mu}}\eta_{\mu
\overline{\beta}}-R_{\mu\overline{\beta}}\eta_{\alpha\overline{\mu}}.
\]

By straightforward calculation, we have
\[%
\begin{array}
[c]{cl}
& \left(  \Delta_{b}-\frac{\partial}{\partial t}\right)  \left\Vert
\eta_{\alpha\overline{\beta}}\left(  x,t\right)  \right\Vert ^{2}\\
= & \left(  \eta_{\alpha\overline{\beta}\gamma}\eta_{\overline{\alpha}\beta
}+\eta_{\alpha\overline{\beta}}\eta_{\overline{\alpha}\beta\gamma}\right)
_{\overline{\gamma}}+conj.+\frac{\partial}{\partial t}\left(  \eta
_{\alpha\overline{\beta}}\eta_{\overline{\alpha}\beta}\right) \\
= & 2\left(  \left\vert \eta_{\alpha\overline{\beta}\gamma}\right\vert
^{2}+\left\vert \eta_{\alpha\overline{\beta}\overline{\gamma}}\right\vert
^{2}\right)  +\eta_{\overline{\alpha}\beta}\left(  \Delta_{b}-\frac{\partial
}{\partial t}\right)  \eta_{\alpha\overline{\beta}}+\eta_{\alpha
\overline{\beta}}\left(  \Delta_{b}-\frac{\partial}{\partial t}\right)
\eta_{\overline{\alpha}\beta}\\
= & 2\left(  \left\vert \eta_{\alpha\overline{\beta}\gamma}\right\vert
^{2}+\left\vert \eta_{\alpha\overline{\beta}\overline{\gamma}}\right\vert
^{2}\right)  +4\left(  R_{\delta\overline{\gamma}\alpha\overline{\beta}}%
\eta_{\gamma\overline{\delta}}\eta_{\overline{\alpha}\beta}-R_{\alpha
\overline{\mu}}\eta_{\mu\overline{\beta}}\eta_{\overline{\alpha}\beta}\right)
\\
\geq & 2\left(  \left\vert \eta_{\alpha\overline{\beta}\gamma}\right\vert
^{2}+\left\vert \eta_{\alpha\overline{\beta}\overline{\gamma}}\right\vert
^{2}\right) \\
\geq & 4\left\vert \frac{\eta_{\alpha\overline{\beta}\gamma}\eta
_{\overline{\alpha}\beta}+\eta_{\alpha\overline{\beta}\overline{\gamma}}%
\eta_{\overline{\alpha}\beta}}{2\left\vert \eta_{\alpha\overline{\beta}%
}\right\vert }\right\vert ^{2}\\
= & 4\left\vert \left\vert \eta_{\alpha\overline{\beta}}\right\vert _{\gamma
}\right\vert ^{2}\\
= & 2\left\vert \nabla_{b}\left\Vert \eta_{\alpha\overline{\beta}}\left(
x,t\right)  \right\Vert \right\vert ^{2}.
\end{array}
\]

This is to say that
\[
\left(  \Delta_{b}-\frac{\partial}{\partial t}\right)  \left\Vert \eta
_{\alpha\overline{\beta}}\left(  x,t\right)  \right\Vert \geq0.
\]

It is clear that $\left\Vert \eta_{\alpha\overline{\beta}}\left(  x,t\right)
\right\Vert -h\left(  x,t\right)  $ is a subsolution to the CR heat equation.
It follows from (\ref{m4}) and \cite[Lemma 4.5]{ccf} that
\[
\left\Vert \eta_{\alpha\overline{\beta}}\left(  x,t\right)  \right\Vert \leq
h\left(  x,t\right)  .
\]

\end{proof}

\begin{lemma}
\label{l3} Let $\left(  M,J,\theta\right)  $ be a complete noncompact Sasakian
$\left(  2n+1\right)  $-manifold of nonnegative pseudohermitian Ricci
curvature with
\[
S=%
%TCIMACRO{\dint \limits_{M}}%
%BeginExpansion
{\displaystyle\int\limits_{M}}
%EndExpansion
\exp\left(  -ar^{2}\left(  x\right)  \right)  \left\Vert \eta_{\alpha
\overline{\beta}}\left(  x,0\right)  \right\Vert d\mu\left(  x\right)
<+\infty
\]
for any positive number $a$. Here $\eta$ is chosen as in Lemma \ref{l2}. Then
there is a positive function $\tau\left(  R\right)  $ with $\tau\left(
R\right)  \rightarrow0^{+}$ as $R$ goes to infinity, such that, for any
positive number $T$,
\[
h\left(  x,t\right)  \leq\tau\left(  R\right)
\]
on $A\left(  o;\frac{R}{2},R\right)  \times\left[  0,T\right]  .$ Here
$h\left(  x,t\right)  $ is chosen as in (\ref{m5}).
\end{lemma}

\begin{proof}
Fix $T>0$ and let $R$ $\left(  \geq\sqrt{T}\right)  $ be a sufficiently large
positive number (may assume supp($u$) $\subseteq B\left(  o,\frac{R}%
{8}\right)  $). Let $x\in A\left(  o;\frac{R}{2},R\right)  $. We have, for a
sufficiently small $a>0$,%
\[%
\begin{array}
[c]{ccl}%
h\left(  x,t\right)  & = &
%TCIMACRO{\dint \limits_{B\left(  o,\frac{R}{4}\right)  }}%
%BeginExpansion
{\displaystyle\int\limits_{B\left(  o,\frac{R}{4}\right)  }}
%EndExpansion
H\left(  x,y,t\right)  \left\Vert \eta_{\alpha\overline{\beta}}\left(
y,0\right)  \right\Vert d\mu\left(  y\right) \\
& \leq & \left(  \underset{y\in B\left(  o,\frac{R}{4}\right)  }{\sup}H\left(
x,y,t\right)  \right)
%TCIMACRO{\dint \limits_{B\left(  o,\frac{R}{4}\right)  }}%
%BeginExpansion
{\displaystyle\int\limits_{B\left(  o,\frac{R}{4}\right)  }}
%EndExpansion
\left\Vert \eta_{\alpha\overline{\beta}}\left(  y,0\right)  \right\Vert
d\mu\left(  y\right) \\
& \leq & \frac{C\left(  n\right)  }{V_{x}\left(  \sqrt{t}\right)  }\left(
\underset{y\in B\left(  o,\frac{R}{4}\right)  }{\sup}\exp\left(  -C_{5}\left(
n\right)  \frac{d_{cc}^{2}\left(  x,y\right)  }{t}\right)  \right) \\
&  & \times\exp\left(  \frac{a}{16}R^{2}\right)
%TCIMACRO{\dint \limits_{M}}%
%BeginExpansion
{\displaystyle\int\limits_{M}}
%EndExpansion
\exp\left(  -ar^{2}\left(  x\right)  \right)  \left\Vert \eta_{\alpha
\overline{\beta}}\left(  x,0\right)  \right\Vert d\mu\left(  y\right) \\
& \leq & C\left(  n\right)  \left(  \frac{R}{\sqrt{t}}\right)  ^{2C_{9}\left(
n\right)  }\frac{S}{V_{o}\left(  R\right)  }\exp\left(  -C_{5}\left(
n\right)  \frac{R^{2}}{16t}+\frac{a}{16}R^{2}\right) \\
& = & \tau\left(  R\right) \\
& \longrightarrow & 0^{+}%
\end{array}
\]
as $R\longrightarrow+\infty$. Here we utilize the CR heat kernel estimate, the
facts that $d_{cc}\left(  x,y\right)  \geq\frac{R^{2}}{16}$, and the CR volume
doubling property%
\[
V_{o}\left(  R\right)  \leq V_{x}\left(  3R\right)  \leq C\left(  n\right)
V_{x}\left(  R\right)  .
\]

\end{proof}

\begin{lemma}
\label{l4} Let $\left(  M,J,\theta\right)  $ be a complete noncompact Sasakian
$\left(  2n+1\right)  $-manifold of nonnegative pseudohermitian Ricci
curvature and $v\left(  x,t\right)  $ be the solution to the CR heat equation
with the initial condition $v\left(  x,0\right)  =u\left(  x\right)  $ for
$u\in C_{c}^{\infty}\left(  M\right)  ,$ $u_{0}=0$, and $\eta_{\alpha
\overline{\beta}}=v_{\alpha\overline{\beta}}+v_{\overline{\beta}\alpha}$. Then%
\[
\underset{r\longrightarrow+\infty}{\lim\inf}\int_{0}^{T}%
%TCIMACRO{\dint \limits_{B_{o}\left(  r\right)  }}%
%BeginExpansion
{\displaystyle\int\limits_{B_{o}\left(  r\right)  }}
%EndExpansion
\exp\left(  -ar^{2}\left(  x\right)  \right)  \left\Vert \eta_{\alpha
\overline{\beta}}\left(  x,t\right)  \right\Vert ^{2}dxdt<+\infty.
\]

\end{lemma}

\begin{proof}
Because
\[
v\left(  x,t\right)  =%
%TCIMACRO{\dint \limits_{M}}%
%BeginExpansion
{\displaystyle\int\limits_{M}}
%EndExpansion
H\left(  x,y,t\right)  u\left(  y\right)  d\mu\left(  y\right)
\]
and $u\in C_{c}^{\infty}\left(  M\right)  $, we see that $\left\vert v\left(
x,t\right)  \right\vert \leq C$ in $M\times\left[  0,+\infty\right)  $. By the
fact that%
\[
\left(  \Delta_{b}-\frac{\partial}{\partial t}\right)  v^{2}=2\left\vert
\nabla_{b}v\right\vert ^{2}%
\]
and the help of the cut-off function $\varphi\in C_{c}^{\infty}\left(
M\right)  $ satisfying $0\leq\varphi\leq1,$ $\varphi|_{B_{o}\left(  r\right)
}=1,$ $\varphi|_{M\backslash B_{o}\left(  2r\right)  }=0,$ and $\left\vert
\Delta_{b}\varphi\left(  x\right)  \right\vert \leq\frac{1}{r^{2}\left(
x\right)  }$ for any positive number $r$, it is clear that, for any $r\geq1$,%
\begin{equation}%
\begin{array}
[c]{ccl}%
\int_{0}^{T}(\frac{1}{V_{o}\left(  r\right)  }\int_{B_{o}\left(  r\right)
}\left\vert \nabla_{b}v\right\vert ^{2}d\mu)dt & \leq & C\left(  n\right)
\left(  \frac{1}{r^{2}}\int_{0}^{T}(\frac{1}{V_{o}\left(  r\right)  }%
\int_{B_{o}\left(  2r\right)  }v^{2}d\mu)dt+\frac{1}{V_{o}\left(  r\right)
}\int_{B_{o}\left(  r\right)  }u^{2}d\mu\right) \\
& \leq & C\left(  n\right)  \left(  T+1\right)  .
\end{array}
\label{m6}%
\end{equation}
From the CR B\^{o}chner formula and $u_{0}=0$ (this implies that $v_{0}=0$),
we obtain
\[
\left(  \Delta_{b}-\frac{\partial}{\partial t}\right)  \left\vert \nabla
_{b}v\right\vert ^{2}\geq2\left\vert Hess_{b}v\right\vert ^{2}.
\]

With the similar method as precedes, we have, for any $r\geq1$,
\[%
\begin{array}
[c]{ccl}%
\int_{0}^{T}(\frac{1}{V_{o}\left(  r\right)  }\int_{B_{o}\left(  r\right)
}\left\vert Hess_{b}v\right\vert ^{2}d\mu)dt & \leq & C\left(  n\right)
\left(  \frac{1}{r^{2}}\int_{0}^{T}(\frac{1}{V_{o}\left(  r\right)  }%
\int_{B_{o}\left(  2r\right)  }\left\vert \nabla_{b}v\right\vert ^{2}%
d\mu)dt+\frac{1}{V_{o}\left(  r\right)  }\int_{B_{o}\left(  r\right)
}\left\vert \nabla_{b}u\right\vert ^{2}d\mu\right) \\
& \leq & C\left(  n\right)  \left(  T+1\right)  .
\end{array}
\]

It enables us to deduce that
\[
\underset{r\longrightarrow+\infty}{\lim\inf}\int_{0}^{T}%
%TCIMACRO{\dint \limits_{B_{o}\left(  r\right)  }}%
%BeginExpansion
{\displaystyle\int\limits_{B_{o}\left(  r\right)  }}
%EndExpansion
\exp\left(  -ar^{2}\left(  x\right)  \right)  \left\Vert \eta_{\alpha
\overline{\beta}}\left(  x,t\right)  \right\Vert ^{2}dxdt<+\infty.
\]

\end{proof}

\begin{remark}
We could drop the assumption of $u_{0}=0$ by estimating the integral
\[
\int_{0}^{T}(\frac{1}{V_{o}\left(  r\right)  }\int_{B_{o}\left(  r\right)
}\left\vert \left\langle J\nabla_{b}v,\nabla_{b}v_{0}\right\rangle \right\vert
^{2}d\mu)dt\leq\left\Vert \nabla_{b}v\right\Vert _{L^{2}}\left\Vert \nabla
_{b}v_{0}\right\Vert _{L^{2}}\leq C\left(  n\right)  \left(  T+1\right)
\]
by (\ref{m6}).
\end{remark}

Now, we are going to prove Theorem \ref{t1} as follows.

\begin{proof}
We will apply the maximum principle to confirm the validity of this theorem.
The detail is as below. For any small positive number $\epsilon>0$, define%
\[
\widetilde{\eta}_{\alpha\overline{\beta}}\left(  x,t\right)  =\eta
_{\alpha\overline{\beta}}\left(  x,t\right)  +\left(  \epsilon\phi\left(
x,t\right)  -\lambda\left(  x,t\right)  \right)  h_{\alpha\overline{\beta}%
}\left(  x\right)  ,
\]
where
\[
\phi\left(  x,t\right)  =e^{t}%
%TCIMACRO{\dint \limits_{M}}%
%BeginExpansion
{\displaystyle\int\limits_{M}}
%EndExpansion
H\left(  x,y,t\right)  \exp\left(  r\left(  y\right)  \right)  d\mu\left(
y\right)  .
\]

We observe that%
\[
\widetilde{\eta}_{\alpha\overline{\beta}}\left(  x,0\right)  >0
\]
for any $x\in M.$ By Lemma \ref{l2} and Lemma \ref{l3},
\[
\widetilde{\eta}_{\alpha\overline{\beta}}\left(  x,t\right)  >0
\]
on $\partial B_{o}\left(  R\right)  \times\left[  0,T\right]  $ for any
sufficiently large number $R>0$. Suppose there is a point $\left(  x_{0}%
,t_{0}\right)  \in B_{o}\left(  R\right)  \times\left(  0,T\right]  $ such
that%
\[
\widetilde{\eta}_{\alpha\overline{\beta}}\left(  x_{0},t_{0}\right)  <0.
\]
There is a number $t_{1}\in\left[  0,t_{0}\right)  $ such that%
\[
\widetilde{\eta}_{\alpha\overline{\beta}}\left(  x,t\right)  \geq0
\]
on $B_{o}\left(  R\right)  \times\left(  0,t_{1}\right]  $ and the bottom
spectrum of $\widetilde{\eta}_{\alpha\overline{\beta}}\left(  x_{1}%
,t_{1}\right)  $
\[
\lambda_{1}\left(  \widetilde{\eta}_{\alpha\overline{\beta}}\left(
x_{1},t_{1}\right)  \right)  =0
\]
for some $x_{1}\in B_{o}\left(  R\right)  $. We assume that $\widetilde{\eta
}_{\alpha\overline{\beta}}\left(  x_{1},t_{1}\right)  $ is a diagonal matrix
and $\widetilde{\eta}_{\gamma\overline{\gamma}}\left(  x_{1},t_{1}\right)  =0$
for $\gamma\in T_{x_{1}}^{1,0}(M)$ with $\left\Vert \gamma\right\Vert =1$.
Here the vector field $\gamma$ is chosen as one of the basis of the CR
structure $T_{x_{1}}^{1,0}(M)$. Therefore, we have%
\begin{equation}
\left(  \frac{\partial}{\partial t}-\Delta_{b}\right)  \widetilde{\eta
}_{\gamma\overline{\gamma}}\left(  x_{1},t_{1}\right)  \leq0. \label{m7}%
\end{equation}
By the fact that $\eta_{\alpha\overline{\beta}}$ is a solution to the CR
Lichnerowicz-Laplace heat equation, we see
\[
\left(  \frac{\partial}{\partial t}-\Delta_{b}\right)  \eta_{\gamma
\overline{\gamma}}\geq0.
\]
Furthermore, with the help of
\[
\left(  \frac{\partial}{\partial t}-\Delta_{b}\right)  \phi=\phi,
\]
we obtain
\[
\left(  \frac{\partial}{\partial t}-\Delta_{b}\right)  \eta_{\gamma
\overline{\gamma}}=\epsilon\phi h_{\gamma\overline{\gamma}}>0
\]
at the point $\left(  x_{1},t_{1}\right)  $. However, it contradicts with the
inequality (\ref{m7}). \ Let $R$ go to infinity, and then $\epsilon$ go to
zero. The proof is completed.
\end{proof}

\end{document}